\documentclass[11pt,a4paper]{amsart}
\usepackage[margin=28mm]{geometry}
\usepackage[T1]{fontenc}
\usepackage[utf8]{inputenc}
\usepackage{lmodern}
\usepackage{amsmath,amssymb,amsthm}
\usepackage{graphicx,tikz}
\usetikzlibrary{positioning}
\usepackage{booktabs,longtable,array}
\usepackage{float,placeins}
\usepackage{needspace}
\usepackage{etoolbox,enumitem}
\usepackage[expansion=false]{microtype}
\usepackage{hyperref}
\hypersetup{hidelinks,
 pdftitle={Contact Rigidity and Comparison Kernels for Type A lci Schubert Varieties},
 pdfauthor={Minghua Dou}}
\newtheorem{theorem}{Theorem}[section]
\newtheorem{lemma}[theorem]{Lemma}
\newtheorem{proposition}[theorem]{Proposition}
\newtheorem{corollary}[theorem]{Corollary}
\theoremstyle{definition}
\newtheorem{definition}[theorem]{Definition}
\newtheorem{example}[theorem]{Example}
\theoremstyle{remark}
\newtheorem{remark}[theorem]{Remark}
\numberwithin{equation}{section}
\setlist{itemsep=2pt,parsep=0pt,topsep=5pt}
\allowdisplaybreaks[1]
\AtBeginEnvironment{thebibliography}{\interlinepenalty=10000\relax}

\newcommand{\AuthorName}{Minghua Dou}
\newcommand{\AuthorDepartment}{Department of Mathematics, School of Natural Sciences}
\newcommand{\AuthorInstitution}{The University of Manchester}
\newcommand{\AuthorCity}{Manchester}
\newcommand{\AuthorCountry}{United Kingdom}
\newcommand{\AuthorEmail}{}
\newcommand{\FundingStatement}{}
\newcommand{\CompetingInterestsStatement}{}

\title[Contact Rigidity and Comparison Kernels]{Contact Rigidity and Comparison Kernels for Type A lci Schubert Varieties}
\author{\AuthorName}
\address{\AuthorDepartment, \AuthorInstitution, \AuthorCity, \AuthorCountry}
\ifdefempty{\AuthorEmail}{}{\email{\AuthorEmail}}
\subjclass[2020]{14M15, 05E14, 14B05, 32S60}
\keywords{Schubert varieties, singular loci, Bruhat order, permutation patterns, perverse sheaves, Kazhdan--Lusztig polynomials}
\date{}
\begin{document}
\begin{abstract}
Let \(X_w\) be a Type~A local complete intersection Schubert variety.
We prove Contact Rigidity: the existence of two smooth singular components
forces some pair of singular components to contain a common Schubert
subvariety of codimension one in each. If the singular locus is a
single smooth component \(X_z\), the rational comparison kernel is
\(IC_z\), and
\(P_{u,w}(q)=1+q^{(\ell(w)-\ell(z)-1)/2}\) for every \(u\leq z\).
The first proof combines pattern avoidance with a computer-assisted
finite overlap classification, rectangle inheritance, and extremal
repairs. The second establishes the hypotheses of Woo's theorem and
uses Euler characteristics and Bruhat triangularity to identify the
entire perverse kernel.
\end{abstract}
\maketitle

\section{Introduction}
\label{chap:introduction}

Schubert varieties provide a setting in which singularities and their
cohomological invariants admit explicit combinatorial descriptions.
For a permutation \(w\in S_n\), let
\(X_w\subseteq\operatorname{GL}_n(\mathbb C)/B\) be the Schubert
variety indexed by \(w\), where \(B\) is the upper-triangular Borel
subgroup. Its Bruhat cells are indexed by \(u\leq w\). Kazhdan--Lusztig
polynomials were introduced in \cite{KL79}; their geometric
interpretation encodes the local intersection cohomology of Schubert
varieties \cite{BL00,Ach21}. We study two questions for local complete
intersection (lci) Schubert varieties: how their singular components
intersect, and how a smooth irreducible singular locus determines the
kernel of the comparison with intersection cohomology.

\subsection{Context and main results}
\label{sec:intro-comparison-kernel}
\label{sec:intro-main-results}

Billey--Warrington and Woo describe the maximal singular indices
through specified pattern
occurrences and conditions on the surrounding permutation graph
\cite[Theorem~1]{BW03}, \cite[Theorem~2.1 and Section~2.4]{Woo09}.
The smoothness criterion of Lakshmibai--Sandhya
\cite{LS90} and the lci criterion of
\'{U}lfarsson--Woo \cite[Theorems~4.8 and~5.2]{UW13} similarly characterise
these properties by pattern avoidance. Our results use these
classifications to obtain a uniform incidence restriction and a
sheaf-theoretic comparison on an entire singular component.

Write \(M_w=\operatorname{maxsing}(w)\), so that
\(\operatorname{Sing}(X_w)=\bigcup_{z\in M_w}X_z\).
All Schubert varieties contain the standard flag; the question is
therefore how large intersections of singular components must be.
We say that two components are \emph{in contact} if they contain a
common Schubert subvariety of codimension one in each. In Bruhat order,
this means that their indices have a common lower cover. We write
\(y\lessdot z\) for \(y<z\) with \(\ell(z)=\ell(y)+1\).

\begin{theorem}[Contact Rigidity]
\label{thm:intro-contact-rigidity}
Let \(w\in S_n\), and assume that \(X_w\) is lci. Suppose that there
are distinct \(z_1,z_2\in M_w\) such that \(X_{z_1}\) and
\(X_{z_2}\) are smooth. Then there are distinct \(z,z'\in M_w\) and
\(y\in S_n\) such that \(y\lessdot z\) and \(y\lessdot z'\).
Consequently, \(X_y\subseteq X_z\cap X_{z'}\), and \(X_y\) has
codimension one in both \(X_z\) and \(X_{z'}\).
\end{theorem}

The theorem requires only two smooth components. The pair in the
conclusion need not be the prescribed pair; the theorem does not assert
that both resulting components are smooth. The examples in Section~\ref{sec:contact-proof}
show that a pairwise strengthening fails, that the lci hypothesis
cannot be omitted, and that one smooth component does not suffice.
For an illustration, \(X_{34512}\) has the three smooth singular
components \(X_{13524}\), \(X_{14325}\), and \(X_{31425}\), each of
dimension three; all contain the two-dimensional \(X_{13425}\)
(Figure~\ref{fig:34512-pentagon}).

Our second question concerns the difference between ordinary and
intersection cohomology. Work throughout with rational coefficients,
and normalise \(IC_w\) by
\(IC_w|_{C_w}\simeq\mathbb Q_{C_w}[\ell(w)]\). Since \(X_w\) is
lci, the shifted constant sheaf is perverse, and the canonical
comparison morphism gives an exact sequence of perverse sheaves
\cite[Remark~1.1]{Hepler19}:
\begin{equation}
  0\longrightarrow N_w\longrightarrow\mathbb Q_{X_w}[\ell(w)]
  \longrightarrow IC_w\longrightarrow0.
  \label{eq:intro-comparison-sequence}
\end{equation}
The support of \(N_w\) is precisely \(\operatorname{Sing}(X_w)\)
(Proposition~\ref{prop:comparison-support}). Contact Rigidity
restricts the incidence of the components of this support. When the
support is smooth and irreducible, we obtain the following identification.

\begin{theorem}[Single-branch comparison]
\label{thm:intro-single-branch-comparison}
Let \(w\in S_n\), and assume that \(X_w\) is lci. Suppose that
\(\operatorname{Sing}(X_w)=X_z\) for some \(z\leq w\), and that
\(X_z\) is smooth. Then \(N_w\simeq IC_z\). Here \(IC_z\) is
regarded as a perverse sheaf on \(X_w\) through the closed embedding
\(X_z\hookrightarrow X_w\).
\end{theorem}

Here \emph{single-branch} refers to the irreducibility of the entire
singular locus. The conclusion identifies the whole kernel, including
its behaviour on the boundary of \(C_z\); its generic restriction
alone does not determine it. The numerical input is Woo's constancy
theorem, which applies under additional pattern-avoidance conditions
\cite[Theorem~1.1 and Corollary~1.2]{Woo09}. We prove these conditions
from the geometric hypotheses and then exclude additional composition
factors on smaller Bruhat strata. The precise polynomial follows.

\begin{corollary}[Single-branch Kazhdan--Lusztig constancy]
\label{cor:intro-single-branch-kl-constancy}
Under the hypotheses of Theorem~\ref{thm:intro-single-branch-comparison},
\[
 P_{u,w}(q)=P_{z,w}(q)=1+q^{(\ell(w)-\ell(z)-1)/2}
 \qquad (u\leq z).
\]
The difference \(\ell(w)-\ell(z)\) is odd and at least three.
\end{corollary}

Combining the two theorems also determines \(N_w\) when all singular
components are rationally smooth and distinct components meet in
codimension at least two in each
(Theorem~\ref{thm:type-a-comparison}). Rational smoothness equals
smoothness in Type~A \cite[ADE-Theorem]{CK03}, so Contact Rigidity
forces at most one component. This consequence does not address a
genuinely reducible singular locus.

\subsection{Proof strategy and organisation}
\label{sec:intro-context-organisation}

Section~\ref{chap:background} fixes the notation and recalls the
required classification and sheaf-theoretic facts.
Section~\ref{chap:contact-rigidity} specialises the
singular-component classification to admissible \(P\)- and \(Q\)-cores.
A finite signature reduction makes their overlap analysis independent
of the ambient rank. Rectangle inheritance and two extremal repairs
control the graph points omitted by standardisation; Cross Completion
and the Contact Pentagon then produce a common lower cover.
Appendix~\ref{app:finite-certificates} retains the complete finite
certificates, and Appendix~\ref{app:contact-code} describes the
ancillary program and recorded output.

Section~\ref{chap:comparison-application} proves single-branch
comparison independently of Contact Rigidity. Guarded pattern
obstructions establish the hypotheses of Woo's theorem. Fixed-point
Euler characteristics and Bruhat triangularity determine \([N_w]\)
using only the values at \(q=1\). The resulting class is one simple
class with multiplicity one, which identifies \(N_w\) itself; the
precise exponent is recovered afterwards.

\section{Preliminaries}
\label{chap:background}

We fix the Schubert, pattern, and sheaf-theoretic conventions used below.
Standard references are \cite{BL00,BB05,Ach21}; admissible cores are
introduced in Section~\ref{sec:cores-overlap}.

\subsection{Schubert varieties and Bruhat combinatorics}
\label{sec:schubert-bruhat}

Let \(G=\operatorname{GL}_n(\mathbb C)\), let \(B\) be the
upper-triangular Borel subgroup, and let \(T\subset B\) be the diagonal
torus. We identify \(G/B\) with the complete flag variety in
\(\mathbb C^n\), with standard flag
\(E_p=\operatorname{span}(e_1,\ldots,e_p)\), and identify the Weyl
group with \(S_n\). For \(w\in S_n\), the permutation matrix
\(\dot w\) satisfies \(\dot w e_i=e_{w(i)}\). Write
\begin{equation}
    G/B=\bigsqcup_{w\in S_n}C_w,
    \qquad C_w=B\dot wB/B,
    \qquad X_w=\overline{C_w}.
    \label{eq:bruhat-decomposition}
\end{equation}
The fixed point \(e_w=\dot wB/B\) corresponds to the flag
\(F_p^w=\operatorname{span}(e_{w(1)},\ldots,e_{w(p)})\).
With \(\leq\) denoting strong Bruhat order, we have
\begin{equation}
    C_w\cong\mathbb C^{\ell(w)},
    \qquad \dim X_w=\ell(w),
    \qquad X_u\subseteq X_w\Longleftrightarrow u\leq w,
    \label{eq:schubert-dictionary}
\end{equation}
and
\begin{equation}
    X_w=\bigsqcup_{u\leq w}C_u;
    \label{eq:schubert-cell-closure}
\end{equation}
see \cite[pp.~325--326]{Ach21}.

We use one-line notation \(w=w_1\cdots w_n\), where \(w_i=w(i)\),
and compose permutations by
\begin{equation}
    (xy)(r)=x\bigl(y(r)\bigr).
    \label{eq:permutation-composition}
\end{equation}
Thus right multiplication by \(t_{ij}=(i,j)\) exchanges the entries
in positions \(i,j\); we write \(s_i=t_{i,i+1}\).

\begin{definition}[Inversion length]
\label{def:inversion-length}
For \(w\in S_n\), set
\[
    \operatorname{Inv}(w)=\{(i,j):i<j,\ w_i>w_j\},
    \qquad \ell(w)=\#\operatorname{Inv}(w).
\]
\end{definition}
This is the Coxeter length. Bruhat order is characterised by the subword
criterion and is graded by \(\ell\) \cite{BB05}.

\begin{definition}[Lower cover]
\label{def:bruhat-lower-cover}
We write \(u\lessdot v\) if \(u<v\) in strong Bruhat order and
\(\ell(v)=\ell(u)+1\).
\end{definition}
The geometric interpretation is
\begin{equation}
    u\lessdot v\quad\Longrightarrow\quad
    X_u\subset X_v,
    \qquad\operatorname{codim}_{X_v}X_u=1.
    \label{eq:cover-codimension-one}
\end{equation}

The permutation plot is \(G(w)=\{(i,w_i):1\leq i\leq n\}\), with
positions horizontal and values vertical. For \(a<b\) and \(r<s\), set
\(\mathcal R(a,b;r,s)=(a,b)\times(r,s)\) and
\begin{equation}
    R_w(a,b;r,s)
    =G(w)\cap\mathcal R(a,b;r,s)
    =\{(p,w_p):a<p<b,\ r<w_p<s\}.
    \label{eq:open-rectangle}
\end{equation}
All boundaries are excluded. The following standard criterion is used
throughout the contact argument.

\begin{lemma}[Bruhat transposition-cover criterion]
\label{lem:bruhat-transposition-cover}
Let \(v\in S_n\), \(i<j\), and \(v_i>v_j\). Then
\begin{equation}
    \ell(v)-\ell(vt_{ij})
    =1+2\#\{k:i<k<j,\ v_j<v_k<v_i\}.
    \label{eq:length-drop-transposition}
\end{equation}
Consequently,
\begin{equation}
    vt_{ij}\lessdot v
    \quad\Longleftrightarrow\quad
    R_v(i,j;v_j,v_i)=\varnothing.
    \label{eq:cover-empty-rectangle}
\end{equation}
\end{lemma}

\begin{proof}
The swap removes the inversion \((i,j)\) and two further inversions
for each \(k\) with \(i<k<j\) and \(v_j<v_k<v_i\); all other
contributions cancel. The reflection criterion gives \(vt_{ij}<v\),
so this is a cover exactly when the length drops by one \cite{BB05}.
\end{proof}

A point of \(R_v(i,j;v_j,v_i)\) is called a \emph{blocker} for the
downward transposition \(t_{ij}\).

\subsection{Singular loci, patterns, and the lci condition}
\label{sec:singular-patterns-lci}

\begin{definition}[Flattening and classical pattern containment]
\label{def:classical-pattern}
For distinct integers \(a_1,\ldots,a_r\), their \emph{flattening} is
the permutation \(\operatorname{fl}(a_1,\ldots,a_r)\in S_r\) given by
\[
    \operatorname{fl}(a_1,\ldots,a_r)_i
    =1+\#\{j:a_j<a_i\}.
\]
A permutation \(w\in S_n\) \emph{contains} \(p\in S_r\) if
\(\operatorname{fl}(w_{i_1},\ldots,w_{i_r})=p\) for some
\(i_1<\cdots<i_r\). These indices form an \emph{occurrence}; if
none exists, \(w\) \emph{avoids} \(p\).
\end{definition}
Pattern avoidance is hereditary under passage to a flattened
subsequence. In particular, a forbidden pattern in a standardised
configuration is also a forbidden pattern in the ambient permutation.

\begin{theorem}[Lakshmibai--Sandhya {\cite{LS90}}]
\label{thm:lakshmibai-sandhya}
The Schubert variety \(X_w\) is smooth if and only if \(w\) avoids
\(3412\) and \(4231\).
\end{theorem}

A variety is a \emph{local complete intersection} (lci) if it is locally
cut out in a smooth ambient variety by the expected number of equations;
see \cite[Section~2.1]{UW13}.

\begin{theorem}[\'{U}lfarsson--Woo {\cite[Theorems~4.8 and~5.2]{UW13}}]
\label{thm:lci-pattern-criterion}
The Schubert variety \(X_w\) is lci if and only if \(w\) avoids
\[
    53241,\quad52341,\quad52431,\quad35142,\quad42513,\quad351624.
\]
\end{theorem}

We call \(w\) an \emph{lci permutation} when \(X_w\) is lci.

\begin{remark}[The final forbidden pattern]
\label{rem:lci-pattern-discrepancy}
Theorems~4.8 and~5.2 of \cite{UW13} give the avoidance criterion
with \(351624\). The introductory Theorem~1.1 in the preprint
\href{https://arxiv.org/abs/1111.6146v1}{arXiv:1111.6146v1}
instead prints \(426153\). We use the criterion proved in
Theorems~4.8 and~5.2.
\end{remark}

The closed singular locus of \(X_w\) is \(B\)-stable, hence is a
union of Schubert varieties. Set
\[
    M_w=\operatorname{maxsing}(w)
    :=\max_{\leq}\{z\leq w:e_z\in\operatorname{Sing}(X_w)\}.
\]
Then its irreducible components are precisely the \(X_z\) with
\(z\in M_w\), and
\begin{equation}
    \operatorname{Sing}(X_w)=\bigcup_{z\in M_w}X_z.
    \label{eq:singular-locus-components}
\end{equation}
Billey--Warrington classify these indices by specified \(4231\)-,
\(3412\)-, and \(45312\)-occurrences with surrounding-region
conditions \cite[Theorem~1]{BW03}; Woo gives an equivalent interval
pattern formulation \cite[Theorem~2.1 and Section~2.4]{Woo09}.
The additional conditions are essential for maximality. Their lci
specialisation is recorded by admissible \(P\)- and \(Q\)-cores in
Section~\ref{sec:cores-overlap} and Proposition~\ref{prop:core-dictionary}.

For a pure \(d\)-dimensional complex variety, rational smoothness at
\(x\) means that a neighbourhood \(U\) is a rational homology
manifold: for every \(y\in U\),
\[
    H^i(U,U\setminus\{y\};\mathbb Q)
    \cong
    \begin{cases}
      \mathbb Q,&i=2d,\\
      0,&i\ne2d.
    \end{cases}
\]
See \cite[Theorem~2.1]{Hepler19}. In Type~A, at every \(x\in X_w\),
\begin{equation}
    x\text{ is rationally smooth in }X_w
    \quad\Longleftrightarrow\quad
    x\text{ is smooth in }X_w.
    \label{eq:type-a-rational-smoothness}
\end{equation}
Indeed, the Peterson ADE theorem gives the nontrivial implication at
fixed points \cite[ADE-Theorem (Second Version)]{CK03},
and both properties are constant along each Bruhat cell. The theorem
applies via
\(\operatorname{SL}_n/(B\cap\operatorname{SL}_n)\cong G/B\).

\subsection{Perverse sheaves and Kazhdan--Lusztig theory}
\label{sec:perverse-kl}

We work with rational coefficients in the classical topology. Write
\(D_c^b(X;\mathbb Q)\) for the constructible derived category and
\(\operatorname{Perv}(X;\mathbb Q)\) for the heart of its middle
perverse \(t\)-structure. Shifts are cohomological, and
\(i_x^*\mathcal F\) is the derived stalk at \(x\). Throughout,
support means \emph{closed support}:
\[
    \operatorname{supp}\mathcal F
    =\overline{\{x:H^j(i_x^*\mathcal F)\ne0
                    \text{ for some }j\}}.
\]
Thus a point of the support need not have a nonzero stalk; see
\cite[Section~1.1, pp.~2 and~4]{Ach21}.

For an irreducible complex variety \(X\) of dimension \(d\), with
smooth-locus inclusion \(j:X_{\mathrm{reg}}\hookrightarrow X\), use
the perverse normalisation
\begin{equation}
    IC_X=j_{!*}\mathbb Q_{X_{\mathrm{reg}}}[d].
    \label{eq:intersection-cohomology-complex}
\end{equation}
Intermediate extension has no nonzero perverse subobject or quotient
supported on the boundary \cite[Lemma~3.3.2]{Ach21}. If \(X\) is
smooth, \(IC_X=\mathbb Q_X[d]\); write \(IC_w=IC_{X_w}\), so the
shift for a Schubert variety is \(\ell(w)\).

Suppose \(\mathbb Q_X[d]\) is perverse, as holds when \(X\) is lci
\cite[p.~1605]{Hepler19}. The comparison morphism
\begin{equation}
    c_X:\mathbb Q_X[d]\longrightarrow IC_X
    \label{eq:comparison-morphism}
\end{equation}
restricts to the identity on \(X_{\mathrm{reg}}\)
\cite[Remark~1.1]{Hepler19}. Its perverse cokernel vanishes by the
boundary-quotient property of intermediate extension. Defining
\(N_X=\ker(c_X)\) in the perverse category gives
\begin{equation}
    0\longrightarrow N_X
    \longrightarrow\mathbb Q_X[d]
    \xrightarrow{\,c_X\,}IC_X
    \longrightarrow0.
    \label{eq:comparison-sequence}
\end{equation}

For an open subset \(U\subseteq X\), restriction is perverse
\(t\)-exact and commutes with intermediate extension, so
\(IC_X|_U\cong IC_U\); see \cite[Lemma~3.1.4(1), Definition~3.3.1,
Theorem~1.2.13, and Proposition~1.2.16]{Ach21}. Consequently,
\[
    N_X|_U=0
    \quad\Longleftrightarrow\quad
    \mathbb Q_U[d]\cong IC_U
    \quad\Longleftrightarrow\quad
    U\text{ is a rational homology manifold}.
\]
For the reverse direction of the first equivalence, identify the source
with \(IC_U\); the kernel is a boundary-supported subobject and hence
vanishes. The second equivalence is \cite[Theorem~2.1]{Hepler19}.
Our closed-support convention therefore gives
\begin{equation}
    \operatorname{supp}N_X
    =\{x\in X:X\text{ is not rationally smooth at }x\}.
    \label{eq:support-comparison-kernel}
\end{equation}

For an lci Schubert variety, write \(N_w=N_{X_w}\). Then
\begin{equation}
    0\longrightarrow N_w
    \longrightarrow\mathbb Q_{X_w}[\ell(w)]
    \longrightarrow IC_w
    \longrightarrow0,
    \label{eq:schubert-comparison-sequence}
\end{equation}
and \eqref{eq:type-a-rational-smoothness} gives
\(\operatorname{supp}N_w=\operatorname{Sing}(X_w)\).

Let \(u\leq w\) and \(i_u:\{e_u\}\hookrightarrow X_w\).
The geometric Kazhdan--Lusztig formula, in our normalisation, is
\begin{equation}
    P_{u,w}(q)
    =\sum_{r\geq0}
       \dim H^{-\ell(w)+2r}(i_u^*IC_w)\,q^r,
    \label{eq:kl-normalization}
\end{equation}
with \(H^j(i_u^*IC_w)=0\) unless
\(j\equiv\ell(w)\pmod2\); see \cite[Corollaries~7.3.7
and~7.3.9]{Ach21}. Stalk ranks are constant along \(C_u\).
For \(\mathcal F\in D_c^b(X_w;\mathbb Q)\), set
\begin{equation}
    \chi_u(\mathcal F)
    =\sum_{j\in\mathbb Z}(-1)^j\dim H^j(i_u^*\mathcal F).
    \label{eq:fixed-point-euler-character}
\end{equation}
Thus \(\chi_u(IC_w)=(-1)^{\ell(w)}P_{u,w}(1)\). Additivity along
the triangle associated with \eqref{eq:schubert-comparison-sequence}
gives
\begin{equation}
    \chi_u(N_w)
    =(-1)^{\ell(w)}\bigl(1-P_{u,w}(1)\bigr).
    \label{eq:euler-character-comparison-kernel}
\end{equation}

Finally, let
\(\mathcal P_w=\operatorname{Perv}_{\mathrm{Bruhat}}(X_w;\mathbb Q)\)
be the category of perverse sheaves constructible with respect to
\(X_w=\bigsqcup_{u\leq w}C_u\).
Since each Bruhat cell is an affine space, its irreducible local systems
are constant of rank one. Hence \(\mathcal P_w\) is a finite-length
abelian category with simple objects \(IC_v\), \(v\leq w\), viewed
on \(X_w\) by closed pushforward
\cite{BBD82}, \cite[Theorems~1.7.9 and~3.4.5]{Ach21}.
For lci \(X_w\), the comparison sequence is Bruhat-constructible, so
\(N_w\in\mathcal P_w\); see \cite[Lemma~2.3.13(1)]{Ach21}.
The Grothendieck group therefore has the basis
\[
    K_0(\mathcal P_w)
    \cong\bigoplus_{v\leq w}\mathbb Z[IC_v],
\]
and an object's coefficients are its Jordan--H\"older multiplicities
\cite[Theorem~A.3.23 and Definition~A.9.3]{Ach21}.
Equality in \(K_0\) records these multiplicities, but in general does
not determine extension classes. Section~\ref{chap:comparison-application}
uses fixed-point Euler characteristics and Bruhat triangularity to
determine \([N_w]\).

\section{Proof of Contact Rigidity}
\label{chap:contact-rigidity}

We prove Theorem~\ref{thm:intro-contact-rigidity} by constructing a
common lower cover for some pair of maximal singular indices.
Section~\ref{sec:cores-overlap} specialises the classification of
Section~\ref{sec:singular-patterns-lci} to admissible \(P\)- and
\(Q\)-cores and excludes disjoint and one-point overlap for cores with
smooth lowerings. Three-point overlap is handled by the Contact
Pentagon; for two-point overlap, Cross Completion produces another
admissible core sharing three points with an original core.

The computer-assisted part comprises exhaustive one-point and two-point
signature classifications. We prove their finite reduction and
completeness, then check the constructions against graph points omitted
by standardisation. These ambient checks, including the extremal
repairs, establish the result for arbitrary \(n\); enumeration of ambient
permutations in small symmetric groups is not used.

\subsection{Admissible cores and finite overlap reduction}
\label{sec:cores-overlap}

We specialise the region and reducedness conditions in the
Billey--Warrington--Woo classification using the lci criterion. The
surviving configurations are the admissible \(P\)- and \(Q\)-cores
defined below.

\paragraph*{Admissible \(P\)- and \(Q\)-cores.}

We identify a core both with its ordered four-tuple of positions and with
the corresponding set of four points of \(G(w)\). This convention makes
expressions such as \(|C_1\cap C_2|\) unambiguous.
The named rectangles below are geometric open subsets of the plane.
Calling one of them empty means that it contains no point of \(G(w)\),
not that the geometric rectangle itself is the empty set.

\begin{definition}[Admissible \(P\)-core]
\label{def:p-core}
Let \(w\in S_n\). A \emph{\(P\)-core} is a four-tuple
    $C=(a,b,c,d)$,
    $a<b<c<d$,
such that
    $w_d<w_b<w_c<w_a$.
Thus the four selected values have order type \(4231\). Define the three
required open regions
\begin{align*}
    P_L(C)&=\mathcal R(a,b;w_b,w_a),\\
    P_C(C)&=\mathcal R(b,c;w_d,w_a),\\
    P_R(C)&=\mathcal R(c,d;w_d,w_c).
\end{align*}

The core \(C\) is \emph{admissible} if
    $G(w)\cap
    \bigl(P_L(C)\cup P_C(C)\cup P_R(C)\bigr)
    =\varnothing$.
Its \emph{lowering} \(z(C)\in S_n\) is obtained by setting
   $ (z_a,z_b,z_c,z_d)=(w_b,w_d,w_a,w_c)$,
and leaving every other position unchanged. On the four core positions,
this is the replacement
    $4231\longmapsto2143$.
\end{definition}

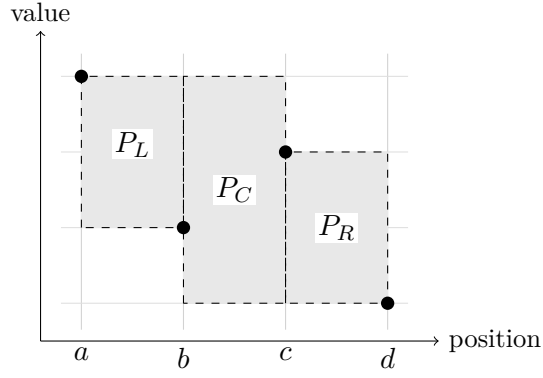
\begin{figure}[!htbp]
    \centering
    \begin{tikzpicture}[x=1.35cm,y=1.0cm]
        \draw[->] (0.6,0.5)--(4.5,0.5) node[right]{\small position};
        \draw[->] (0.6,0.5)--(0.6,4.6) node[above]{\small value};
        \foreach \x/\lab in {1/a,2/b,3/c,4/d}{
            \draw[gray!35] (\x,0.65)--(\x,4.3);
            \node[below] at (\x,0.55) {$\lab$};
        }
        \foreach \y in {1,2,3,4}{\draw[gray!25] (0.8,\y)--(4.2,\y);}
        \fill[gray!18] (1,2) rectangle (2,4);
        \fill[gray!18] (2,1) rectangle (3,4);
        \fill[gray!18] (3,1) rectangle (4,3);
        \draw[dashed] (1,2) rectangle (2,4);
        \draw[dashed] (2,1) rectangle (3,4);
        \draw[dashed] (3,1) rectangle (4,3);
        \node[fill=white,inner sep=1pt] at (1.5,3.1) {$P_L$};
        \node[fill=white,inner sep=1pt] at (2.5,2.5) {$P_C$};
        \node[fill=white,inner sep=1pt] at (3.5,2.0) {$P_R$};
        \foreach \x/\y in {1/4,2/2,3/3,4/1}{\fill (\x,\y) circle (2.5pt);}
    \end{tikzpicture}
    \caption{The schematic form of a \(P\)-core. The four marked points have
order type \(4231\), and the three shaded open regions are
\(P_L(C)\), \(P_C(C)\), and \(P_R(C)\).}
    \label{fig:p-core}
\end{figure}

\Needspace{20\baselineskip}
\begin{definition}[Admissible \(Q\)-core]
\label{def:q-core}
Let \(w\in S_n\). A \emph{\(Q\)-core} is a four-tuple
    $C=(a,b,c,d)$,
    $a<b<c<d$,
such that
    $w_c<w_d<w_a<w_b$.
Thus the four selected values have order type \(3412\). Define the six
required open regions
\begin{align*}
    Q_1(C)&=\mathcal R(a,b;w_d,w_a),
    &Q_2(C)&=\mathcal R(b,c;w_a,w_b),\\
    Q_3(C)&=\mathcal R(c,d;w_d,w_a),
    &Q_4(C)&=\mathcal R(b,c;w_c,w_d),\\
    Q_{A_1}(C)&=\mathcal R(a,b;w_c,w_d),
    &Q_{A_2}(C)&=\mathcal R(c,d;w_a,w_b),
\end{align*}
and the middle region
    $B(C)=\mathcal R(b,c;w_d,w_a)$.
The core \(C\) is \emph{admissible} if
\[
    G(w)\cap
    \bigl(
        Q_1(C)\cup Q_2(C)\cup Q_3(C)\cup Q_4(C)
        \cup Q_{A_1}(C)\cup Q_{A_2}(C)
    \bigr)
    =\varnothing,
\]
and the graph points in \(G(w)\cap B(C)\) are strictly decreasing from
left to right: whenever
\[
    p<q,
    \qquad
    (p,w_p),(q,w_q)\in G(w)\cap B(C),
\]
one has \(w_p>w_q\). Its lowering \(z(C)\in S_n\) is obtained by setting
    $(z_a,z_b,z_c,z_d)=(w_c,w_a,w_d,w_b)$
and leaving every other position unchanged. On the four core positions,
this is
    $3412\longmapsto1324$.
\end{definition}

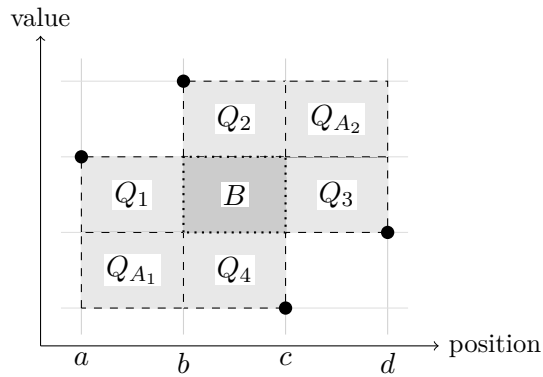
\begin{figure}[!htbp]
    \centering
    \begin{tikzpicture}[x=1.35cm,y=1.0cm]
        \draw[->] (0.6,0.5)--(4.5,0.5) node[right]{\small position};
        \draw[->] (0.6,0.5)--(0.6,4.6) node[above]{\small value};
        \foreach \x/\lab in {1/a,2/b,3/c,4/d}{
            \draw[gray!35] (\x,0.65)--(\x,4.3);
            \node[below] at (\x,0.55) {$\lab$};
        }
        \foreach \y in {1,2,3,4}{\draw[gray!25] (0.8,\y)--(4.2,\y);}
        \foreach \xa/\xb/\ya/\yb in {
            1/2/2/3,2/3/3/4,3/4/2/3,
            2/3/1/2,1/2/1/2,3/4/3/4}{
            \fill[gray!18] (\xa,\ya) rectangle (\xb,\yb);
            \draw[dashed] (\xa,\ya) rectangle (\xb,\yb);
        }
        \fill[gray!40] (2,2) rectangle (3,3);
        \draw[thick,dotted] (2,2) rectangle (3,3);
        \node[fill=white,inner sep=0.8pt] at (1.5,2.5) {$Q_1$};
        \node[fill=white,inner sep=0.8pt] at (2.5,3.5) {$Q_2$};
        \node[fill=white,inner sep=0.8pt] at (3.5,2.5) {$Q_3$};
        \node[fill=white,inner sep=0.8pt] at (2.5,1.5) {$Q_4$};
        \node[fill=white,inner sep=0.8pt] at (1.5,1.5) {$Q_{A_1}$};
        \node[fill=white,inner sep=0.8pt] at (3.5,3.5) {$Q_{A_2}$};
        \node[fill=white,inner sep=0.8pt] at (2.5,2.5) {$B$};
        \foreach \x/\y in {1/3,2/4,3/1,4/2}{\fill (\x,\y) circle (2.5pt);}
    \end{tikzpicture}
    \caption{The schematic form of a \(Q\)-core. The four marked points have
order type \(3412\). The six lightly shaded regions are required to
contain no graph points, while the graph points in the darker middle
region \(B(C)\) must decrease from left to right.}
    \label{fig:q-core}
\end{figure}

\paragraph*{The lci reduction of the maximal-component classification.}

We use Woo's interval-pattern formulation of the singular-locus
classification \cite[Sections~2.3--2.4 and Theorem~2.1]{Woo09}.
For \(x\leq v\) in \(S_m\), write
\([x,v]=\{p\in S_m:x\leq p\leq v\}\).
An interval pattern embedding \([x,v]\hookrightarrow[u,w]\) uses
the same positions for occurrences of \(x\) in \(u\) and \(v\) in
\(w\), requires \(u\) and \(w\) to agree outside those positions,
and induces an isomorphism of the two Bruhat intervals. Under these
endpoint conditions, the interval-isomorphism requirement is equivalent
to \(\ell(v)-\ell(x)=\ell(w)-\ell(u)\); see
\cite[Section~2.1 and Lemma~2.1]{WY08}.

The Billey--Warrington classification, in Woo's interval-pattern
formulation, states that \(X_w\) is singular at \(e_{u'}\) precisely
when a model interval embeds in \([u,w]\) for some \(u'\leq u<w\)
\cite[Theorem~1]{BW03}, \cite[Theorem~2.1 and Section~2.4]{Woo09}.
Here we include the \([14325,45312]\) boundary interval, as explained
in Remark~\ref{rem:type-IIB-boundary} below. The corresponding \(X_u\)
are exactly the irreducible singular components. Type~I is based on a
\(4231\)-configuration; Types~IIA and~IIB use a critical
\(3412\)-configuration. We use the equivalent permutation-graph
conditions from \cite[Section~2.4]{Woo09}.

We begin with Type~I. For integers \(y,t>0\), set
\begin{align*}
    x^{\mathrm I}_{y,t}
    &=\bigl(y+1,y,\ldots,1,\,
      y+t+2,y+t+1,\ldots,y+2\bigr),\\
    v^{\mathrm I}_{y,t}
    &=\bigl(y+t+2,\,
      y+1,y,\ldots,2,\,
      y+t+1,y+t,\ldots,y+2,\,1\bigr).
\end{align*}
In these model permutations, a decreasing string whose initial entry
is smaller than its final entry is empty.

A Type~I component \(X_u\) is specified by an embedding of
\([x^{\mathrm I}_{y,t},v^{\mathrm I}_{y,t}]\) in \([u,w]\).
For \(y=t=1\), the model is \([2143,4231]\); the lci hypothesis excludes
all larger models. If the embedding positions are
\[
    i=j_0<j_1<\cdots<j_y<k_1<\cdots<k_t<k_{t+1}=m,
\]
Woo's graph criterion requires the regions
\[
    (j_{r-1},j_r)\times(w_{j_r},w_i)\quad(1\leq r\leq y),
    \qquad
    (k_s,k_{s+1})\times(w_m,w_{k_s})\quad(1\leq s\leq t),
\]
and \((j_y,k_1)\times(w_m,w_i)\) to contain no point of \(G(w)\)
\cite[Section~2.4]{Woo09}. For \(y=t=1\), these are the three defining
rectangles of a \(P\)-core.

\begin{lemma}[Type~I reduction and its \(P\)-core form]
\label{lem:lci-type-I-reduction}
Assume that \(X_w\) is lci, and let \(X_u\) be a Type~I irreducible
component of \(\operatorname{Sing}(X_w)\).  Then its Type~I parameters
satisfy
    $y=t=1$.
Hence the model interval defining \(X_u\) is the minimal Type~I interval
    $[2143,4231]$.

Let
$    a<b<c<d$
be the four positions of the corresponding interval embedding in \(w\).
Then the values at these positions have relative order \(4231\), and
Woo's graph conditions for this interval embedding are equivalent to
\[
    G(w)\cap
\bigl(P_L(C)\cup P_C(C)\cup P_R(C)\bigr)
=\varnothing,
    \qquad C=(a,b,c,d).
\]
Consequently \(C\) is an admissible \(P\)-core, and its lowering is the
permutation indexing the component:
    $u=z(C)$.
Thus every Type~I component of an lci Schubert variety is encoded by an
admissible \(P\)-core.
\end{lemma}

\begin{proof}
Set \(N=y+t+2\). The upper endpoint is
\[
    v^{\mathrm I}_{y,t}
    =(N,\,y+1,y,\ldots,2,\,N-1,N-2,\ldots,y+2,\,1).
\]
If \(y\geq2\), the entries \(N,y+1,y,N-1,1\) form \(53241\),
which would also occur in \(w\), contrary to
Theorem~\ref{thm:lci-pattern-criterion}. Thus \(y=1\).
If \(t\geq2\), the entries \(N,2,N-1,N-2,1\) then form the forbidden
pattern \(52431\). Since \(y,t>0\), this gives \(y=t=1\) and hence
\(x^{\mathrm I}_{1,1}=2143\), \(v^{\mathrm I}_{1,1}=4231\).

Write the embedding indices as
\(a=i<b=j_1<c=k_1<d=m=k_2\), so \(w_d<w_b<w_c<w_a\).
Woo's required empty regions become
\[
    (a,b)\times(w_b,w_a),\qquad
    (b,c)\times(w_d,w_a),\qquad
    (c,d)\times(w_d,w_c),
\]
which are exactly \(P_L(C),P_C(C),P_R(C)\) in
Definition~\ref{def:p-core}. Thus \(C\) is admissible.
The interval embedding replaces \(4231\) by \(2143\) on these same four
positions and fixes all others, so its lower endpoint is \(u=z(C)\).
\end{proof}

For a Type~II component, choose its associated reduced critical
\(3412\)-embedding at \(a<b<c<d\), with \(w_c<w_d<w_a<w_b\).
Woo's four required empty critical regions are
\((a,b)\times(w_d,w_a)\), \((b,c)\times(w_a,w_b)\),
\((c,d)\times(w_d,w_a)\), and \((b,c)\times(w_c,w_d)\)
\cite[Section~2.4]{Woo09}. These are precisely
\(Q_1(C),Q_2(C),Q_3(C),Q_4(C)\).
The auxiliary and middle regions are
\begin{align*}
    A_1&=(a,b)\times(w_c,w_d)=Q_{A_1}(C),\\
    A_2&=(c,d)\times(w_a,w_b)=Q_{A_2}(C),
\end{align*}
and \(B=(b,c)\times(w_d,w_a)=B(C)\).
Reducedness means that the graph points in \(B\) decrease from left to
right \cite[Section~2.4, p.~7]{Woo09}; it is a separate condition and
does not follow from critical emptiness. The lci hypothesis supplies the
two remaining emptiness conditions.

\begin{lemma}[The lci condition forces \(A_1\) and \(A_2\) to be empty]
\label{lem:lci-a-regions}
Let \(w\in S_n\) be such that \(X_w\) is lci, and let
\(C=(a,b,c,d)\) be any \(3412\)-occurrence in \(w\).
Set \(A_1=Q_{A_1}(C)\) and \(A_2=Q_{A_2}(C)\), using
Definition~\ref{def:q-core}. Then
\[
    G(w)\cap A_1=\varnothing,
    \qquad
    G(w)\cap A_2=\varnothing.
\]
\end{lemma}

\begin{proof}
If \((p,w_p)\in G(w)\cap A_1\), then
\(a<p<b\) and \(w_c<w_p<w_d<w_a<w_b\).
Thus the positions \(a,p,b,c,d\), with values
\(w_a,w_p,w_b,w_c,w_d\), form \(42513\), contradicting
Theorem~\ref{thm:lci-pattern-criterion}.
Similarly, \((p,w_p)\in G(w)\cap A_2\) gives
\(c<p<d\) and \(w_c<w_d<w_a<w_p<w_b\), so the positions
\(a,b,c,p,d\), with values \(w_a,w_b,w_c,w_p,w_d\), form the forbidden
pattern \(35142\).
\end{proof}

In Woo's graph description of Type~II components, the case
\(G(w)\cap B=\varnothing\) is called Type~IIA
\cite[Section~2.4 and Theorem~2.1]{Woo09}.
The corresponding model intervals are parametrised by two
nonnegative integers \(y,t\).  Their upper endpoints, written in
one-line notation, are
\[
    y+3,\ y+1,\ldots,2,\ N,\ 1,\ N-1,\ldots,y+4,\ y+2,
    \qquad N=y+t+4.
\]
The next lemma shows that the lci condition forces both parameters
to vanish.
\begin{lemma}[Type~IIA components in the lci case]
\label{lem:lci-type-IIA-reduction}
Assume that \(X_w\) is lci.  If a maximal singular component is of
Woo's Type~IIA, then the parameters \(y,t\) in the corresponding
Type~IIA model satisfy
    $y=t=0$.
Consequently, the corresponding Bruhat interval is
    $[1324,3412]$.
\end{lemma}

\begin{proof}
Let
    $N=y+t+4$.
The upper endpoint of the Type~IIA model is
\[
    v=
    \bigl(
    y+3,\ y+1,\ldots,2,\ N,\ 1,\,
    N-1,\ldots,y+4,\ y+2
    \bigr).
\]

If \(y\geq1\), the five displayed entries
    $y+3, y+1, N, 1, y+2$
occur in this order and have relative order \(42513\).
Thus \(w\) contains \(42513\), contradicting
Theorem~\ref{thm:lci-pattern-criterion}.

Similarly, if \(t\geq1\), the entries
    $y+3, N, 1,N-1,y+2$
have relative order \(35142\).  This is again forbidden by
Theorem~\ref{thm:lci-pattern-criterion}.

Therefore
   $ y=t=0$.
For these parameter values the Type~IIA model is the interval
    $[1324,3412]$.
\end{proof}

\begin{remark}[How \(Q\)-cores unify the Type~II cases]
\label{rem:type-IIB-boundary}

Write the critical values of \(C=(a,b,c,d)\) as
\(\gamma=w_c<\delta=w_d<\alpha=w_a<\beta=w_b\).
The middle region is \(B(C)=\mathcal R(b,c;\delta,\alpha)\).
For a reduced Type~II embedding, its graph points
\((p_1,\eta_1),\ldots,(p_r,\eta_r)\), ordered by
\(p_1<\cdots<p_r\), satisfy
\(\delta<\eta_r<\cdots<\eta_1<\alpha\) when \(r\geq1\)
\cite[Section~2.4]{Woo09}. On the positions
\(a<b<p_1<\cdots<p_r<c<d\), the lowering is
\[
    (\alpha,\beta,\eta_1,\ldots,\eta_r,\gamma,\delta)
    \longmapsto
    (\gamma,\alpha,\eta_1,\ldots,\eta_r,\delta,\beta).
\]
It changes only the four critical values and fixes the middle chain.

For \(r=0\), this gives the minimal Type~IIA interval \([1324,3412]\).
For \(r=1\), writing \(\eta=\eta_1\), the inequalities
\(\gamma<\delta<\eta<\alpha<\beta\) identify the upper values
\((\alpha,\beta,\eta,\gamma,\delta)\) with \(45312\) and the lower
values \((\gamma,\alpha,\eta,\delta,\beta)\) with \(14325\).
This is the boundary interval \([14325,45312]\) in
\cite[Theorem~1]{BW03}. For \(r\geq2\), one obtains Woo's Type~IIB
family \cite[Theorem~2.1(IIB) and Section~2.4]{Woo09}, with the same
lowering and unchanged middle chain. After
Lemma~\ref{lem:lci-a-regions} removes \(A_1,A_2\), all three cases are
therefore encoded by Definition~\ref{def:q-core}.
\end{remark}

We collect the preceding reductions in the following dictionary.
Here and below, an \emph{admissible core} means either an admissible
\(P\)-core or an admissible \(Q\)-core.

\begin{proposition}[Core dictionary]
\label{prop:core-dictionary}
Assume that \(X_w\) is lci. Then:
\begin{enumerate}
    \item every \(z\in\operatorname{maxsing}(w)\) is the lowering of a
    unique admissible \(P\)-core or \(Q\)-core;
    \item for every admissible core \(C\), one has
        $z(C)\in\operatorname{maxsing}(w)$.
\end{enumerate}
Thus admissible cores and the irreducible components of
\(\operatorname{Sing}(X_w)\) are in bijection.
\end{proposition}

\begin{proof}
For Type~I components, Woo's classification
\cite[Theorem~2.1(I) and Section~2.4]{Woo09}, together with
Lemma~\ref{lem:lci-type-I-reduction}, leaves only the interval
    $[2143,4231]$.
Its graph conditions are exactly the three required empty regions of an
admissible \(P\)-core.

For Type~II components, Woo's graph description supplies four critical
empty regions, the auxiliary regions \(A_1,A_2\), the middle region
\(B\), and the reducedness condition
\cite[Section~2.4]{Woo09}. Lemma~\ref{lem:lci-a-regions} removes
\(A_1\) and \(A_2\). Lemma~\ref{lem:lci-type-IIA-reduction} and
Remark~\ref{rem:type-IIB-boundary} then identify the remaining Type~IIA,
Type~IIB, and boundary \(45312\) cases with admissible \(Q\)-cores.

Conversely, an admissible \(P\)-core satisfies the full graph conditions
for the Type~I interval \([2143,4231]\). An admissible \(Q\)-core is a
reduced critical \(3412\)-embedding with \(G(w)\cap A_1=G(w)\cap A_2=\varnothing\).
If \(G(w)\cap B(C)\) is empty, it gives the Type~IIA case; if it contains
one point, it gives the \(45312\) boundary case; and if it contains at
least two decreasing points, it gives the Type~IIB case. The
Billey--Warrington--Woo classification therefore shows in every case that
   $ z(C)\in\operatorname{maxsing}(w)$.

It remains to prove uniqueness. For either lowering, the set of positions
at which \(z(C)\) differs from \(w\) is exactly the set of four core
positions. Hence, if \(z(C)=z(C')\), the two cores have the same position
set. The restriction of \(w\) to these positions has order type \(4231\)
or \(3412\), which determines whether the core is of type \(P\) or \(Q\).
Thus \(C=C'\).
\end{proof}

\paragraph*{Standardisation and elementary cells.}

Let \(C_1,C_2\) have \(r\leq8\) union points, with ordered positions
\(x_1<\cdots<x_r\) and values \(v_1<\cdots<v_r\).
Their standardisation is the unique \(\pi\in S_r\) satisfying
\(w(x_i)=v_{\pi(i)}\); equivalently,
\(\pi=\operatorname{fl}(w(x_1),\ldots,w(x_r))\), as in
Definition~\ref{def:classical-pattern}.
A core on \(x_{i_1},x_{i_2},x_{i_3},x_{i_4}\), with
\(i_1<i_2<i_3<i_4\), has support
\(I=\{i_1,i_2,i_3,i_4\}\). We denote the core by \(I\!P\) or
\(I\!Q\), according to its type; for example, \(1245Q\) is a \(Q\)-core on
\(x_1,x_2,x_4,x_5\).

For \(1\leq i,j\leq r-1\), define the elementary cell
\begin{equation}
    \square_{ij}
    =(x_i,x_{i+1})\times(v_j,v_{j+1}).
    \label{eq:elementary-cell}
\end{equation}

Any graph point lying inside a rectangle determined by the union points,
but not belonging to the union itself, lies in exactly one such cell.
Indeed, its position cannot equal any selected \(x_i\), and its value
cannot equal any selected \(v_j\), since \(w\) is a permutation.

For a core \(C\), let \(E(C)\) be the collection of elementary cells
contained in the regions that are required to be empty by the
admissibility conditions.  If \(C\) is a \(Q\)-core, let \(H(C)\) be
the collection of cells contained in its middle region \(B(C)\).
For a \(P\)-core, set \(H(C)=\varnothing\).
The standardised union determines these collections of cells and the
admissibility tests on the union points. Ambient graph points in the
cells must be controlled separately.

\begin{lemma}[Standardisation is compatible with lowering]
\label{lem:standardisation-lowering}
Let \(C\) be a core all four of whose graph points belong to a selected
finite union, and let \(\pi\) be the standardised word of that union. Lower \(w\) at \(C\), restrict to the union, and standardise. The
result is the same as applying
    $4231\longmapsto2143$
    or
   $ 3412\longmapsto1324$
directly to the four corresponding positions of \(\pi\).
\end{lemma}

\begin{proof}
Lowering only permutes the four values already occupying the core
positions. It neither introduces nor removes a union value, and every
other union position retains its value. Hence every pairwise comparison
among union values changes exactly as it does when the same four-value
replacement is performed in the rank word \(\pi\).
\end{proof}

\paragraph*{Finite signature reduction.}

We now record the finite conditions satisfied by the standardised union
of two admissible cores with smooth lowerings in an lci permutation.
If the two cores overlap in one or two points, their union has
\(r=7\) or \(r=6\) points, respectively.

Let \(\pi\in S_r\) be the standardised union, and write
\([r]=\{1,\ldots,r\}\). Let
\(I,J\subseteq[r]\) record the four points used by the two cores, and let
\(\tau,\sigma\in\{P,Q\}\) record their types.  We use the following
finite tests:

\begin{itemize}
    \item
    \(\operatorname{Core}(\pi,I,\tau)\) means that the four entries
    indexed by \(I\) have the required order type
    (\(4231\) for \(P\), \(3412\) for \(Q\)), that no other union point
    lies in a region required to be empty, and, in the \(Q\)-case, that
    the union points lying in the middle region decrease from left to
    right.

    \item
    \(\operatorname{Sm}(\pi,I,\tau)\) means that, after lowering the four
    entries indexed by \(I\), the resulting permutation avoids
    \(3412\) and \(4231\).

    \item
    \(\operatorname{LCI}(\pi)\) means that \(\pi\) avoids
    \[
        53241,\ 52341,\ 52431,\ 35142,\ 42513,\ 351624.
    \]
\end{itemize}

These tests record exactly the consequences of admissibility,
smoothness, and the lci hypothesis that are visible from the finite
standardised union.  They are necessary conditions for an actual pair
of cores, but need not be sufficient, since graph points outside the
union are not recorded by \(\pi\).

\begin{theorem}[Finite signature reduction]
\label{thm:finite-signature-reduction}
Let \(C_1,C_2\) be admissible cores in an lci permutation \(w\), and
assume that \(X_{z(C_1)}\) and \(X_{z(C_2)}\) are smooth.  Let
    $s=|C_1\cap C_2|$,
   $ r=8-s$,
so that the union of the two cores contains \(r\) points.

After standardising this union, let \(\pi\in S_r\).  Let
\(I,J\subseteq\{1,\ldots,r\}\) be the four-element subsets recording
which union points belong to \(C_1\) and \(C_2\), respectively, and let
\(\tau,\sigma\in\{P,Q\}\) record their types.  Then
   $ I\cup J=\{1,\ldots,r\}$,
    $|I\cap J|=s$,
and the standardised data satisfy
\[
\begin{gathered}
    \operatorname{Core}(\pi,I,\tau),
    \qquad
    \operatorname{Core}(\pi,J,\sigma),\\
    \operatorname{Sm}(\pi,I,\tau),
    \qquad
    \operatorname{Sm}(\pi,J,\sigma),
    \qquad
    \operatorname{LCI}(\pi).
\end{gathered}
\]

Hence every actual pair of admissible cores satisfying the hypotheses
appears, after standardisation, among the finitely many data
    $(\pi,I,J,\tau,\sigma)$
satisfying the conditions above. This gives a complete finite list of
cases. To obtain an ambient conclusion from this list, any proposed
construction must also be checked against graph points outside the
selected union.
\end{theorem}

\begin{proof}
Admissibility of \(C_1,C_2\) implies that no selected union point lies
in a required empty region. Restricting a decreasing middle chain to
the union preserves its order. Thus both \(\operatorname{Core}\)
conditions hold.

By Lemma~\ref{lem:standardisation-lowering}, lowering \(\pi\) on either
support gives a classical pattern of the corresponding \(z(C_i)\).
Smoothness and Theorem~\ref{thm:lakshmibai-sandhya} imply that
\(z(C_i)\), and therefore this restriction, avoids \(3412\) and
\(4231\). Hence both \(\operatorname{Sm}\) conditions hold.

Finally, \(\pi\) is a classical pattern of \(w\), so it inherits
avoidance of \(53241\), \(52341\), \(52431\), \(35142\), \(42513\), and \(351624\) from the lci
criterion, Theorem~\ref{thm:lci-pattern-criterion}. This proves
\(\operatorname{LCI}(\pi)\). The later rectangle arguments and extremal
repairs check graph points omitted from the union.
\end{proof}

\begin{proposition}[Complete finite enumeration]
\label{prop:finite-decision}
Fix the overlap size \(s\).  Every finite signature appearing in
Theorem~\ref{thm:finite-signature-reduction} can be found by the
following exhaustive search:
\begin{enumerate}
    \item list all unordered pairs \(I,J\subseteq\{1,\ldots,8-s\}\)
    of four-element subsets satisfying
        $I\cup J=\{1,\ldots,8-s\}$,
        $|I\cap J|=s$;
    \item choose independently the types of the two cores,
        $(\tau,\sigma)\in\{P,Q\}^2$;
    \item list every permutation
        $\pi\in S_{8-s}$;
    \item keep exactly those data satisfying the
    \(\operatorname{Core}\), \(\operatorname{Sm}\), and
    \(\operatorname{LCI}\) conditions.
\end{enumerate}
For \(s=1\) and \(s=2\), the numbers of unordered support pairs are
respectively
\[
    \frac12\binom71\binom63=70,
    \qquad
    \frac12\binom62\binom42=45.
\]
The procedure is independent of the ambient rank \(n\).
\end{proposition}

\begin{proof}
A possible support pair consists of two four-element subsets whose union is
\([8-s]\) and whose intersection has size \(s\). Step~1 lists exactly
these unordered pairs. Step~2 lists all possible core types, and Step~3
lists every possible vertical order on the standardised union. Step~4
implements the defining predicates without any additional assumption.
Hence no eligible finite signature is omitted.

The ancillary program described in Appendix~\ref{app:contact-code}
implements this decision procedure without hard-coding the output rows.
It carries out only the finite check whose completeness has been proved
above.
\end{proof}

\paragraph*{Minimum overlap.}

We now exclude disjoint and one-point overlap for cores with smooth
lowerings.

\begin{lemma}[Minimum overlap]
\label{lem:minimum-overlap}
Let \(C_1,C_2\) be distinct admissible cores in an lci permutation
\(w\). If \(X_{z(C_1)}\) and \(X_{z(C_2)}\) are smooth, then
    $|C_1\cap C_2|\geq2$.
\end{lemma}

\begin{proof}
Suppose first that \(C_1\cap C_2=\varnothing\).  By definition, lowering
\(C_1\) changes only the four positions belonging to \(C_1\).  Hence the
four positions of \(C_2\) are unchanged in \(z(C_1)\), and still form a
\(3412\)- or \(4231\)-pattern.  Therefore \(z(C_1)\) contains one of the
two Lakshmibai--Sandhya forbidden patterns, so \(X_{z(C_1)}\) is
singular, a contradiction.

Now suppose that \(|C_1\cap C_2|=1\). The proof of
Theorem~\ref{thm:finite-signature-reduction}, with the
\(\operatorname{LCI}\) predicate omitted, reduces every such ambient pair
to the finite search defined by the \(\operatorname{Core}\) and
\(\operatorname{Sm}\) predicates. Applying the search in
Proposition~\ref{prop:finite-decision} with the
\(\operatorname{LCI}\) test omitted produces exactly eight signatures,
listed with their certificates in
Appendix~\ref{app:finite-certificates}. Each signature contains one of
\[
    351624,\qquad 42513,\qquad 35142
\]
as an explicitly identified classical pattern. Since the standardised
union is itself a classical pattern of \(w\), the same forbidden pattern
occurs in \(w\), contradicting
Theorem~\ref{thm:lci-pattern-criterion}. Hence the overlap is neither zero
nor one.
\end{proof}

\subsection{Contact Pentagon and Cross Completion}
\label{sec:pentagon-cross}

By Lemma~\ref{lem:minimum-overlap}, two distinct cores associated with
smooth maximal components overlap in at least two points. They cannot
overlap in all four points, since uniqueness in
Proposition~\ref{prop:core-dictionary} would make them equal. Thus only
two- and three-point overlap remain. Three-point overlap admits a direct,
computer-free geometric argument; two-point overlap requires a finite
completion step that creates a third core with three-point overlap.

\paragraph*{Inversion symmetry and the deletion table.}

For a permutation \(p=p_1p_2p_3p_4\in S_4\), let
\(\partial_i p\in S_3\) denote the permutation obtained by deleting the
\(i\)-th entry and flattening the remaining three entries. The following symmetry will be used to reduce the number of cases in the
three-point overlap analysis.  In particular, inversion exchanges the
shared patterns \(231\) and \(312\), while preserving all hypotheses and
the common-lower-cover conclusion.

\begin{lemma}[Inversion symmetry]
\label{lem:inversion-symmetry}
The map \(w\mapsto w^{-1}\) (where \(w^{-1}\) denotes the inverse of
\(w\) in \(S_n\))  preserves Bruhat order and length, and hence
preserves lower covers. It transposes the permutation plot and has the
following additional properties:
\begin{enumerate}
    \item admissible \(P\)-cores map to admissible \(P\)-cores, and
    admissible \(Q\)-cores map to admissible \(Q\)-cores;
    \item lowering commutes with inversion;
    \item the patterns \(3412\), \(4231\), and the six lci forbidden
    patterns are preserved as avoidance families under inversion.
\end{enumerate}
\end{lemma}

\begin{proof}
Inversion preserves Coxeter length and Bruhat order; see
\cite[Chapter~2]{BB05}. The patterns \(4231\) and \(3412\) are
self-inverse, so transposition preserves the type of the four selected
points.

To check the empty-region conditions, use the four selected positions
and values to divide the interior of their bounding rectangle into nine
open cells. Label a cell \((i,j)\), where \(i,j\in\{1,2,3\}\)
record its horizontal and vertical intervals. The required empty cells
for a \(P\)-core have labels
\[
    (1,2),\ (1,3),\ (2,1),\ (2,2),\ (2,3),\ (3,1),\ (3,2),
\]
and those for a \(Q\)-core have labels
\[
    (1,1),\ (1,2),\ (2,1),\ (2,3),\ (3,2),\ (3,3).
\]
Both sets of labels are invariant under \((i,j)\mapsto(j,i)\).
No other permutation point lies on a selected horizontal or vertical
coordinate line, and the four selected points lie outside the required
open regions. Thus the empty-region conditions are preserved. This
argument concerns the union of the required cells; an individual named
rectangle need not map to an individual named rectangle.

For a \(Q\)-core, the middle region is the central cell. A decreasing
chain remains decreasing after transposition. The lowering formulas
commute directly with inversion on the four selected values.
Finally, \(3412\), \(4231\), and each of the six lci forbidden
patterns are self-inverse. This proves all three assertions.
\end{proof}

For the two patterns defining \(P\)- and \(Q\)-cores,
we obtain
\begin{equation}
\begin{aligned}
    \bigl(
    \partial_1(4231),\partial_2(4231),
    \partial_3(4231),\partial_4(4231)
    \bigr)
        &=(231,321,321,312),\\
    \bigl(
    \partial_1(3412),\partial_2(3412),
    \partial_3(3412),\partial_4(3412)
    \bigr)
        &=(312,312,231,231).
\end{aligned}
\label{eq:deletion-calculus}
\end{equation}
Hence, if two cores share three points, the shared triple must have
relative order \(231\), \(312\), or \(321\).  Inversion exchanges
\(231\) and \(312\) and fixes \(321\), so it is enough to consider the
cases \(312\) and \(321\).

\paragraph*{Excluding the shared pattern \(321\).}

\begin{lemma}[Exclusion of a shared \(321\)-triple]
\label{lem:shared-321}
If two distinct admissible cores share three points whose relative order
is \(321\), then \(w\) contains \(52341\). Hence this overlap cannot
occur in an lci permutation.
\end{lemma}

\begin{proof}
By \eqref{eq:deletion-calculus}, a shared \(321\)-triple can arise only
from two \(P\)-cores, by deleting the point in role~2 or role~3 from each
core. Write the shared positions as
    $s_1<s_2<s_3$
and their values as
    $H>M>L$.
If both unique points are role~2, they lie in
\((s_1,s_2)\times(L,M)\). Write them in horizontal order as \(p<q\).
Then \((q,w_q)\) lies in the \(P_C\)-rectangle of the core whose unique
role~2 point is \((p,w_p)\), contradicting admissibility.

If both unique points are role~3, they lie in
\((s_2,s_3)\times(M,H)\). For \(p<q\), the point \((p,w_p)\) lies in
the \(P_C\)-rectangle of the core whose unique role~3 point is
\((q,w_q)\), again contradicting admissibility.

The remaining case has one role~2 point \(p\) and one role~3 point \(q\).
Then
   $ s_1<p<s_2<q<s_3$,
    $L<w_p<M<w_q<H$.
At these five positions, the values have relative order
    $H,\ w_p,\ M,\ w_q,\ L\sim52341$.
Thus \(w\) contains \(52341\), which is forbidden when \(X_w\) is
lci by Theorem~\ref{thm:lci-pattern-criterion}.
\end{proof}

\paragraph*{The two pentagon normal forms.}

Assume that the shared triple has pattern \(312\). Write its positions as
\(s_1<s_2<s_3\) and its values as \(H,L,M\), where \(L<M<H\).
Equation~\eqref{eq:deletion-calculus} gives precisely three possible
extensions:
\begin{enumerate}
    \item a \(P\)-core whose unique point occupies role~4 and lies to
    the right of \(s_3\), with value below \(L\);
    \item a \(Q\)-core whose unique point occupies role~1 and lies to
    the left of \(s_1\), with value in \((M,H)\);
    \item a \(Q\)-core whose unique point occupies role~2 and lies in
    \((s_1,s_2)\), with value above \(H\).
\end{enumerate}

\begin{lemma}[Pentagon normal forms]
\label{lem:pentagon-normal-forms}
Let \(C_1,C_2\) be distinct admissible cores in an lci permutation
\(w\), and suppose that \(|C_1\cap C_2|=3\). Up to inversion,
their five-point union has relative order \(34512\) or \(45231\).
\end{lemma}

\begin{proof}
The shared pattern \(321\) is excluded by Lemma~\ref{lem:shared-321},
and inversion reduces shared \(231\) to shared \(312\). We analyse the
three types of extension listed above.

In the \(P/P\) case, both unique points lie to the right of \(s_3\) and
below \(L\). Write them as \(p<q\). If \(w_q<w_p\), then \((p,w_p)\)
lies in the \(P_R\)-rectangle of the core whose role~4 point is
\((q,w_q)\), contrary to admissibility. Thus \(w_p<w_q\), and the five
points have pattern
    $H,L,M,w_p,w_q\sim53412$.
Since \((53412)^{-1}=45231\), inversion gives the second normal form.

In the \(Q/Q\) case, suppose first that both unique points are role~1.
Write \(p<q<s_1\). If \(w_q<w_p\), then \((q,w_q)\) lies in the
\(Q_1\)-rectangle of the core with role~1 point \((p,w_p)\). Hence
\(w_p<w_q\), which gives \(34512\). If both unique points are role~2,
write \(s_1<p<q<s_2\). The inequality \(w_q<w_p\) would place
\((q,w_q)\) in the \(Q_2\)-rectangle of the core whose role~2 point is
\((p,w_p)\); hence again \(w_p<w_q\), giving \(34512\). If one unique
point is role~1 and the other is role~2, their horizontal and vertical
roles determine \(34512\) directly.

In the \(P/Q\) case, the \(P\)-unique point lies at the far right and
below \(L\). If the \(Q\)-unique point is role~1, the five values in
position order are
\[
    \text{a value in }(M,H),\quad H,\quad L,\quad M,\quad
    \text{a value below }L,
\]
which standardise to \(45231\). If the \(Q\)-unique point is role~2, the
values are
\[
    H,\quad \text{a value above }H,\quad L,\quad M,\quad
    \text{a value below }L,
\]
and they again standardise to \(45231\). These cases exhaust all
possibilities.
\end{proof}

\paragraph*{The \(34512\) Contact Pentagon.}

For the first normal form we construct a common lower cover \(y\).
The swaps on the five selected points must satisfy
Lemma~\ref{lem:bruhat-transposition-cover} in the full permutation, so
we check their ambient blocker rectangles.

Let the five union positions and values be
    $x_1<\cdots<x_5$,
    $v_1<\cdots<v_5$,
and set
    $U:=\{x_1,\ldots,x_5\}$.
Assume that
    $\bigl(w(x_1),\ldots,w(x_5)\bigr)
    =
    (v_3,v_4,v_5,v_1,v_2)$,
so that the standardised word on \(U\) is \(34512\).

The possible cores are
    $1245Q$,
    $1345Q$,
    $2345Q$,
and their lowerings have standardised words
    $13524$,
    $14325$,
    $31425$,
respectively.

Define \(y\) to agree with \(w\) outside \(U\). On \(U\), set
\[
    \bigl(y(x_1),\ldots,y(x_5)\bigr)=(v_1,v_3,v_4,v_2,v_5),
\]
so that its standardised word is \(13425\).  On the five
union positions, the three lowerings are connected to \(y\) by
\begin{equation}
    13524
    \xrightarrow{\ t_{x_3,x_5}\ }
    13425,
    \qquad
    14325
    \xrightarrow{\ t_{x_2,x_3}\ }
    13425,
    \qquad
    31425
    \xrightarrow{\ t_{x_1,x_2}\ }
    13425.
    \label{eq:34512-pentagon-swaps}
\end{equation}

We check the swaps corresponding to the two chosen admissible cores.
The third displayed core is not assumed to be admissible in the ambient
permutation. Every graph point outside \(U\) is unchanged by lowering
and avoids all selected coordinate lines. For the first swap, the only
intermediate union position is \(x_4\), whose lowered value is \(v_2\),
so it is not a blocker. The other two swaps have no intermediate union
position. Thus it remains to exclude blockers outside \(U\). Recall from
Lemma~\ref{lem:bruhat-transposition-cover} that a graph point blocks a
downward swap between positions \(x_i<x_j\) only if its position lies
strictly between \(x_i\) and \(x_j\) and its value lies strictly between
the two values being exchanged.

For the first swap
    $13524
    \xrightarrow{\ t_{x_3,x_5}\ }
    13425$,
the exchanged values are \(v_5\) and \(v_4\).  Hence any blocker in the
full permutation would have to lie in
    $(x_3,x_5)\times(v_4,v_5)$.
If the first core is \(1245Q\), the second core is either \(1345Q\) or
\(2345Q\).  In either case, the part of this rectangle over
\((x_3,x_4)\) is contained in a required \(Q_2\)-region of the second
core, while the part over \((x_4,x_5)\) is contained in its required
\(Q_{A_2}\)-region.  Since the second core is admissible, both regions
are empty.  Thus the blocker rectangle is empty, and
Lemma~\ref{lem:bruhat-transposition-cover} gives
    \(y\lessdot z(C)\) in \(S_n\), where \(C\) is the chosen core
with support label \(1245Q\).

For the second swap, the blocker strip is
    $(x_2,x_3)\times(v_3,v_4)$.
It lies in \(Q_2\) of the other core when that core is \(1245Q\), and
in \(Q_1\) when the other core is \(2345Q\).

\begin{figure}[!htbp]
    \centering
    \begin{tikzpicture}[
        node distance=1.4cm and 1.7cm,
        box/.style={draw,rounded corners,inner sep=4pt},
        lower/.style={->,thick},
        core/.style={->,dashed}
    ]
        \node[box] (w) {$34512$};
        \node[box,below left=of w] (a) {$13524$};
        \node[box,below=of w] (b) {$14325$};
        \node[box,below right=of w] (c) {$31425$};
        \node[box,below=2.8cm of w] (y) {$13425$};
        \draw[core] (w)--(a) node[pos=.55,sloped,above]{\scriptsize $1245Q$};
        \draw[core] (w)--(b) node[pos=.55,right]{\scriptsize $1345Q$};
        \draw[core] (w)--(c) node[pos=.55,sloped,above]{\scriptsize $2345Q$};
        \draw[lower] (a)--(y)
    node[midway,left]{\scriptsize $t_{x_3,x_5}$};
\draw[lower] (b)--(y)
    node[midway,right]{\scriptsize $t_{x_2,x_3}$};
\draw[lower] (c)--(y)
    node[midway,right]{\scriptsize $t_{x_1,x_2}$};
    \end{tikzpicture}
    \caption{The \(34512\) Contact Pentagon. Dashed arrows denote the
    three core lowerings and are not asserted to be Bruhat covers. The
    solid arrows show the possible lower-cover relations. For any two
    chosen admissible cores, the corresponding solid arrows are certified
    in the ambient permutation by the blocker argument.}
    \label{fig:34512-pentagon}
\end{figure}
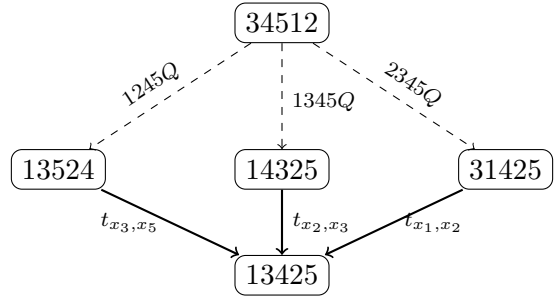

\FloatBarrier

For the third swap, the blocker strip is
    $(x_1,x_2)\times(v_1,v_3)$,
and every possible blocker outside \(U\) lies in
\(Q_{A_1}\cup Q_1\) of the other chosen \(Q\)-core. The omitted
horizontal line has a selected value and contains no point outside
\(U\). Thus no blocker remains. By
Lemma~\ref{lem:bruhat-transposition-cover}, the two arrows associated
with the chosen admissible cores in
\eqref{eq:34512-pentagon-swaps} are ambient Bruhat lower covers.

\paragraph*{The \(45231\) Contact Pentagon.}

Now assume that

    $\bigl(w(x_1),\ldots,w(x_5)\bigr)
    =
    (v_4,v_5,v_2,v_3,v_1)$,
so the standardised word on \(U\) is \(45231\).
The possible cores are
    $1234Q,
    1345P,
    2345P$,
with lowerings
    $24351,
    25143,
    42153$.
Define \(y\) to agree with \(w\) outside \(U\) and to satisfy
    $\bigl(y(x_1),\ldots,y(x_5)\bigr)
    =
    (v_2,v_4,v_1,v_5,v_3)$,
so its standardised word on \(U\) is \(24153\).
On the five union positions, the three descending transpositions are
\begin{equation}
    24351
    \xrightarrow{\ t_{x_3,x_5}\ }
    24153,
    \qquad
    25143
    \xrightarrow{\ t_{x_2,x_4}\ }
    24153,
    \qquad
    42153
    \xrightarrow{\ t_{x_1,x_2}\ }
    24153.
    \label{eq:45231-pentagon-swaps}
\end{equation}

As before, only the arrows for the two chosen admissible cores are
needed. In the first swap the intermediate union point has value
\(v_5\), outside the blocker interval \((v_1,v_3)\). In the second
swap the intermediate union point has value \(v_1\), outside
\((v_4,v_5)\). The third swap has no intermediate union point.
All possible blockers therefore lie outside \(U\) and avoid the
selected coordinate lines.

For the first swap, such a blocker lies in
\((x_3,x_5)\times(v_1,v_3)\), and hence in \(P_C\cup P_R\) of
the other chosen \(P\)-core. For the second swap it lies in
\((x_2,x_4)\times(v_4,v_5)\). This places it in
\(Q_2\cup Q_{A_2}\) if the other core is \(1234Q\), and in
\(P_L\cup P_C\) if the other core is \(2345P\). For the third
swap the blocker strip is \((x_1,x_2)\times(v_2,v_4)\). A blocker
lies in \(Q_{A_1}\cup Q_1\) if the other core is \(1234Q\), and
in \(P_L\) if it is \(1345P\). Each named region contains no graph
point. Lemma~\ref{lem:bruhat-transposition-cover} therefore proves the
two required ambient lower-cover relations in
\eqref{eq:45231-pentagon-swaps}.

\begin{figure}[htbp]
    \centering
    \begin{tikzpicture}[
        node distance=1.4cm and 1.7cm,
        box/.style={draw,rounded corners,inner sep=4pt},
        lower/.style={->,thick},
        core/.style={->,dashed}
    ]
        \node[box] (w) {$45231$};
        \node[box,below left=of w] (a) {$24351$};
        \node[box,below=of w] (b) {$25143$};
        \node[box,below right=of w] (c) {$42153$};
        \node[box,below=2.8cm of w] (y) {$24153$};
        \draw[core] (w)--(a) node[pos=.55,sloped,above]{\scriptsize $1234Q$};
        \draw[core] (w)--(b) node[pos=.55,right]{\scriptsize $1345P$};
        \draw[core] (w)--(c) node[pos=.55,sloped,above]{\scriptsize $2345P$};
        \draw[lower] (a)--(y)
    node[midway,left]{\scriptsize $t_{x_3,x_5}$};
\draw[lower] (b)--(y)
    node[midway,right]{\scriptsize $t_{x_2,x_4}$};
\draw[lower] (c)--(y)
    node[midway,right]{\scriptsize $t_{x_1,x_2}$};
    \end{tikzpicture}
    \caption{The \(45231\) Contact Pentagon, with the same convention as
    Figure~\ref{fig:34512-pentagon}. The solid arrows are precisely the
    common-lower-cover relations used in the proof.}
    \label{fig:45231-pentagon}
\end{figure}
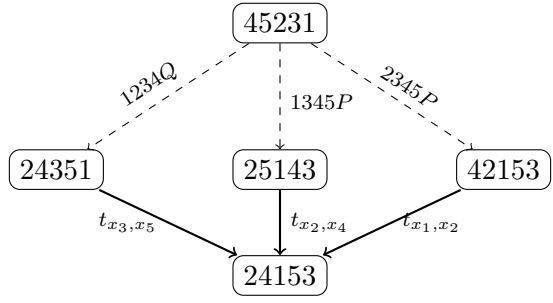

\begin{theorem}[Contact Pentagon Lemma]
\label{thm:contact-pentagon}
Let \(C_1,C_2\) be distinct admissible cores in an lci permutation
\(w\). If
   $ |C_1\cap C_2|=3$,
then there exists \(y\in S_n\) such that
\[
    y\lessdot z(C_1),
    \qquad
    y\lessdot z(C_2).
\]
\end{theorem}

\begin{proof}
The deletion calculus \eqref{eq:deletion-calculus},
Lemma~\ref{lem:shared-321}, Lemma~\ref{lem:pentagon-normal-forms}, and Lemma~\ref{lem:inversion-symmetry} reduce every case to the \(34512\) or \(45231\)
normal form. The preceding two arguments construct the same permutation
\(y\) from either lowering and place every possible ambient blocker in a
named required empty rectangle. Lemma~\ref{lem:bruhat-transposition-cover}
then gives both lower-cover relations.
\end{proof}

Notice that Theorem~\ref{thm:contact-pentagon} does not require the
lowered Schubert varieties to be smooth.

\paragraph*{Two-point Cross Completion.}

It remains to treat \(|C_1\cap C_2|=2\), so the standardised union has six
points. The aim is not to construct a common lower cover directly.
Instead, we construct a third admissible \(Q\)-core sharing three points
with one of the original cores, after which the Contact Pentagon applies.

Set
    $E_{12}=E(C_1)\cup E(C_2)$, 
where \(E(C_i)\) denotes the set of elementary cells contained in the
required empty regions of \(C_i\), as introduced after
\eqref{eq:elementary-cell}.
For a candidate \(Q\)-core \(C\), define
\begin{equation}
    U(C)=E(C)\setminus E_{12},
    \qquad
    H_{\mathrm{res}}(C)=H(C)\setminus E_{12}.
    \label{eq:cross-unresolved-residual}
\end{equation}
If \(U(C)=\varnothing\), every required empty cell of \(C\) is already
forced empty by the admissibility of \(C_1\) and \(C_2\). If, in
addition,
\begin{equation}
    H_{\mathrm{res}}(C)
    \subseteq H(C_1)\cap H(C_2),
    \label{eq:inherited-B-condition}
\end{equation}
and the candidate middle region contains no noncore union point, then
all possible ambient points in its residual middle cells lie in a common
subregion of the two original decreasing chains. The candidate middle
region is therefore reduced.

For each two-point signature with original supports \(I_1,I_2\), choose
a candidate \(Q\)-core on \(K\subseteq\{1,\ldots,6\}\) satisfying
\(\operatorname{Core}(\pi,K,Q)\) and
\(|K\cap I_1|=|K\cap I_2|=3\).
For a support \(K\), write \(U(K)\) and \(H_{\mathrm{res}}(K)\)
for the corresponding sets associated with the candidate core on \(K\).
Ties are resolved in this order:
\begin{enumerate}
    \item minimise \(|U(K)|\);
    \item prefer a candidate whose restricted lowering contains
    \(3412\) or \(4231\);
    \item prefer one satisfying \eqref{eq:inherited-B-condition};
    \item choose the lexicographically smallest support label.
\end{enumerate}
This rule only selects a unique candidate for the finite table.
Ambient admissibility is proved separately by cell inheritance and,
in the two exceptional rows, extremal repair.

\begin{theorem}[Complete two-point classification]
\label{thm:complete-cross-table}
Exactly twenty overlap-two signatures satisfy the finite predicates in
Theorem~\ref{thm:finite-signature-reduction}. For each signature,
Tables~\ref{tab:complete-two-point-cross-structural} and
\ref{tab:complete-two-point-cross-certificates} specify a candidate
\(Q\)-core \(C\) such that
\[
    |C\cap C_1|=|C\cap C_2|=3.
\]
All required cells of \(C\) are inherited from \(C_1\) or \(C_2\), except
for the single cell \(\square_{33}\) in rows~6 and~20. Every residual
middle cell lies in
   $ H(C_1)\cap H(C_2)$,
and no candidate middle region contains a noncore union point.
\end{theorem}

\begin{proof}
Proposition~\ref{prop:finite-decision} enumerates all \(45\) unordered
pairs of position sets, all four type pairs, and all permutations in
\(S_6\). Testing the three finite predicates gives exactly twenty
signatures. The ancillary program and its recorded output are described
in Appendix~\ref{app:contact-code}.

The candidate order type, unresolved cells, and residual middle cells are
read directly from the standardised word and the elementary-cell
decomposition. The complete data are listed in
Tables~\ref{tab:complete-two-point-cross-structural} and
\ref{tab:complete-two-point-cross-certificates}, while the
rectangle-by-rectangle inclusions are recorded in
Appendix~\ref{app:finite-certificates}.
\end{proof}

\begin{lemma}[Rectangle inheritance]
\label{lem:rectangle-inheritance}
For every row of Theorem~\ref{thm:complete-cross-table} except rows~6 and
20, the selected candidate \(C\) is an admissible \(Q\)-core.
\end{lemma}

\begin{proof}
The candidate has order type \(3412\). By the local
\(\operatorname{Core}\) predicate, no noncore union point lies in a
required region. Any graph point in a required or middle region of
\(C\) that is outside the selected union avoids all union position
and value lines, and therefore lies in a unique elementary cell.

Since \(U(C)=\varnothing\), every elementary cell contained in a required
region of \(C\) is contained in a required empty region of \(C_1\) or
\(C_2\). Hence every required region of \(C\) contains no graph point.

The candidate middle region contains no noncore union point. Any other
graph point in that region lies in a residual middle cell belonging to
\(H(C_1)\cap H(C_2)\), and hence lies in \(B(C_1)\cap B(C_2)\).
The original middle chains decrease from left to right. Therefore the
candidate middle region is reduced, and \(C\) is admissible.
\end{proof}

The two exceptional rows each have one potentially nonempty required
cell. An extremal point in that cell replaces one critical point and
produces a new admissible core in a single step.

\begin{lemma}[Extremal repair of row~6]
\label{lem:repair-row-6}
Consider row~6:
\[
    \pi=561234,
    \qquad
    C_1=1236Q,
    \qquad
    C_2=1245Q,
    \qquad
    C=1234Q.
\]
Here the only required cell of the candidate \(C\) whose emptiness is
not inherited from \(C_1\) or \(C_2\) is
    $\square_{33}
    =(x_3,x_4)\times(v_3,v_4)$.
If this cell contains a point of \(G(w)\), then there exists an admissible \(Q\)-core
\(C'\) such that
   $ |C'\cap C_1|=3$.
\end{lemma}

\begin{proof}
Write the six union points as
\[
    (x_1,v_5),\ (x_2,v_6),\ (x_3,v_1),\
    (x_4,v_2),\ (x_5,v_3),\ (x_6,v_4),
\]
where \(x_1<\cdots<x_6\) and \(v_1<\cdots<v_6\). 
If this cell $\square_{33}$ contains no graph point, then \(C\) is already admissible.
Otherwise choose a graph point with minimal horizontal coordinate,
$\xi=(\rho,\tau)$,
   $ x_3<\rho<x_4$,
    $v_3<\tau<v_4$,
and define
    $C'
    =
    \{(x_1,v_5),(x_2,v_6),(x_3,v_1),(\rho,\tau)\}$.
Since
    $v_1<\tau<v_5<v_6$,
these four points have order type \(3412\).

Four required regions are inherited directly:
\begin{align*}
    Q_2(C')
        &=(x_2,x_3)\times(v_5,v_6)=Q_2(C_1),\\
    Q_4(C')
        &=(x_2,x_3)\times(v_1,\tau)\subseteq Q_4(C_1),\\
    Q_{A_1}(C')
        &=(x_1,x_2)\times(v_1,\tau)\subseteq Q_{A_1}(C_1),\\
    Q_{A_2}(C')
        &=(x_3,\rho)\times(v_5,v_6)\subseteq Q_{A_2}(C_1).
\end{align*}

We next prove that \(Q_1(C')\) is empty. Let
\((p,w_p)\in G(w)\cap Q_1(C')\). If \(w_p<v_4\), then
\((p,w_p)\in Q_{A_1}(C_1)\); if \(w_p>v_4\), then
\((p,w_p)\in Q_1(C_1)\). The remaining equality \(w_p=v_4\) is impossible,
because the unique graph point of value \(v_4\) is \((x_6,v_4)\), whereas
\(Q_1(C')\) lies over \((x_1,x_2)\). Hence \(Q_1(C')\) is empty.

Now let \((p,w_p)\in G(w)\cap Q_3(C')\). If \(w_p>v_4\), then
    $(p,w_p)\in Q_3(C_1)$.
If \(w_p<v_4\), then \((p,w_p)\) lies in the unresolved cell strictly to
the left of \(\xi\), contradicting the leftmost choice of \(\xi\).
The equality \(w_p=v_4\) is again impossible, since the unique point of
value \(v_4\) has position \(x_6>\rho\). Thus \(Q_3(C')\) is empty.

Finally,
   $ B(C')=(x_2,x_3)\times(\tau,v_5)$.
A graph point in this region with value below \(v_4\) would lie in the
empty region \(Q_4(C_1)\). Every remaining graph point lies in
\(B(C_1)\), and these points decrease from left to right. The value
\(v_4\) occurs only at position \(x_6\), outside the horizontal interval.
Therefore \(B(C')\) is reduced. Hence \(C'\) is an admissible \(Q\)-core
sharing three points with \(C_1\).
\end{proof}

\begin{lemma}[Extremal repair of row~20]
\label{lem:repair-row-20}
Consider row~20:
\[
    \pi=345612,
    \qquad
    C_1=1456Q,
    \qquad
    C_2=2356Q,
    \qquad
    C=1256Q.
\]
If \(G(w)\cap\square_{33}\neq\varnothing\), then there is an
admissible \(Q\)-core \(C'\) sharing three points with \(C_1\).
\end{lemma}

\begin{proof}
Write the union points as
\[
    (x_1,v_3),\ (x_2,v_4),\ (x_3,v_5),\
    (x_4,v_6),\ (x_5,v_1),\ (x_6,v_2).
\]

If cell $\square_{33}$ contains no graph point, then \(C\) is admissible. Otherwise
choose a graph point with maximal horizontal coordinate,
   $ \xi=(\rho,\tau)$,
    $x_3<\rho<x_4$,
    $v_3<\tau<v_4$,
and define
    $C'
    =
    \{(x_1,v_3),(\rho,\tau),(x_5,v_1),(x_6,v_2)\}$.
Since
    $v_1<v_2<v_3<\tau,$
these four points have order type \(3412\).

Five required regions are inherited directly:
\begin{align*}
    Q_1(C')&\subseteq Q_1(C_1),
    &Q_3(C')&=Q_3(C_1),\\
    Q_4(C')&\subseteq Q_4(C_2),
    &Q_{A_1}(C')&\subseteq Q_{A_1}(C_1),\\
    Q_{A_2}(C')&\subseteq Q_{A_2}(C_1).
\end{align*}

It remains to prove that \(Q_2(C')\) is empty. Let
    $(p,w_p)\in G(w)\cap Q_2(C')$.
If \(p<x_4\), then \((p,w_p)\) lies in the unresolved cell strictly to
the right of \(\xi\), contradicting the rightmost choice of \(\xi\).
If \(p>x_4\), then
    $(p,w_p)\in Q_2(C_1)$.
The remaining case \(p=x_4\) is impossible because the union point at
that position has value \(v_6\), which lies outside the vertical interval
\((v_3,\tau)\). Thus \(Q_2(C')\) is empty.

Finally,
    $B(C')=(\rho,x_5)\times(v_2,v_3)$.
Any graph point in this region with position less than \(x_4\) lies in the
empty region \(Q_1(C_1)\). Any graph point with position greater than \(x_4\)
lies in the reduced middle region \(B(C_1)\). The point at position
\(x_4\) has value \(v_6\), outside the vertical interval. Therefore
\(B(C')\) is reduced, and \(C'\) is an admissible \(Q\)-core sharing
three points with \(C_1\).
\end{proof}

\begin{theorem}[Cross Completion]
\label{thm:cross-completion}
Let \(C_1,C_2\) be distinct admissible cores in an lci permutation
\(w\). Assume that \(X_{z(C_1)}\) and \(X_{z(C_2)}\) are smooth and
that
    $|C_1\cap C_2|=2$.
Then there exists an admissible \(Q\)-core \(C_3\) and
\(i\in\{1,2\}\) such that
    $|C_3\cap C_i|=3$.
\end{theorem}

\begin{proof}
By Theorem~\ref{thm:finite-signature-reduction}, the standardised
signature is one of the twenty rows in
Theorem~\ref{thm:complete-cross-table}. Outside rows~6 and~20,
Lemma~\ref{lem:rectangle-inheritance} gives the required admissible
candidate. In rows~6 and~20, the preliminary candidate is admissible if
its unresolved cell is empty; if it is nonempty, apply respectively
Lemma~\ref{lem:repair-row-6} or Lemma~\ref{lem:repair-row-20}. In every case the selected or repaired core has intersection of
cardinality exactly three with at least one original core.
\end{proof}

\subsection{Contact Rigidity and sharpness}
\label{sec:contact-proof}

The preceding constructions now give the main geometric result. The
examples that follow distinguish its existential conclusion from a
pairwise statement and explain the role of its hypotheses.

\begin{theorem}[Contact Rigidity]
\label{thm:contact-rigidity-proof}
Let \(w\in S_n\), and assume that \(X_w\) is lci. Suppose that there exist
distinct
    $z_1,z_2\in\operatorname{maxsing}(w)$
such that \(X_{z_1}\) and \(X_{z_2}\) are smooth. Then there exist distinct
\(z,z'\in\operatorname{maxsing}(w)\) and \(y\in S_n\) such that
   $ y\lessdot z,
    y\lessdot z'$.
Consequently,
\[
    X_y\subseteq X_z\cap X_{z'},
    \qquad
    \operatorname{codim}_{X_z}X_y
    =
    \operatorname{codim}_{X_{z'}}X_y
    =1.
\]
\end{theorem}

\begin{proof}
By Proposition~\ref{prop:core-dictionary}, there are unique admissible
cores \(C_1,C_2\) such that
    $z(C_i)=z_i
 (i=1,2)$.
Since \(X_{z_1}\) and \(X_{z_2}\) are smooth,
Lemma~\ref{lem:minimum-overlap} gives
    $|C_1\cap C_2|\geq2$.
The cores cannot share all four points, since uniqueness in
Proposition~\ref{prop:core-dictionary} would then give \(C_1=C_2\).
Hence their overlap has size two or three.

If \(|C_1\cap C_2|=3\), Theorem~\ref{thm:contact-pentagon} gives a common
lower cover \(y\) of \(z(C_1)\) and \(z(C_2)\). In this case set
\(z=z(C_1)\) and \(z'=z(C_2)\).

Suppose instead that \(|C_1\cap C_2|=2\). By
Theorem~\ref{thm:cross-completion}, there is an admissible \(Q\)-core
\(C_3\) and \(i\in\{1,2\}\) such that
    $|C_3\cap C_i|=3$.
Proposition~\ref{prop:core-dictionary} gives
    $z(C_3)\in\operatorname{maxsing}(w)$.
The Contact Pentagon Lemma applies directly to \(C_3\) and \(C_i\);
it does not require their lowerings to be smooth. Hence there exists
\(y\in S_n\) such that
    $y\lessdot z(C_3)$,
    $y\lessdot z(C_i)$.
The two maximal singular indices are distinct because their unique cores
share three, rather than four, points. In this case set \(z=z(C_3)\)
and \(z'=z(C_i)\).

Finally, \eqref{eq:cover-codimension-one} gives
\[
    X_y\subseteq X_z\cap X_{z'},
    \qquad
    \operatorname{codim}_{X_z}X_y
    =
    \operatorname{codim}_{X_{z'}}X_y
    =1.
\]
\end{proof}

\paragraph*{Why the theorem is not pairwise.}

\begin{example}[Failure of the pairwise strengthening]
\label{ex:false-pairwise-contact}
Contact Rigidity does not assert that the pair it supplies consists of
smooth components. Nor does every prescribed pair have a common lower
cover, even when all maximal singular components are smooth.

Let
    $w=451623$.
A direct check shows that \(w\) avoids the six forbidden patterns in
Theorem~\ref{thm:lci-pattern-criterion}, so \(X_w\) is lci. The
admissible cores of \(w\) are exactly
\[
\begin{array}{c|ccccc}
C
    &1235Q&1236Q&1256Q&1456Q&2456Q\\
\hline
z(C)
    &142653&143625&241635&251436&421536 .
\end{array}
\]
Hence Proposition~\ref{prop:core-dictionary} gives
    $\operatorname{maxsing}(w)
    =
    \{142653,143625,241635,251436,421536\}$.
Each of these five permutations avoids \(3412\) and \(4231\), so
Theorem~\ref{thm:lakshmibai-sandhya} shows that all five corresponding
Schubert varieties are smooth.

Consider, however, the two maximal singular indices
    $z=142653$,
    $z'=251436$.
Using Lemma~\ref{lem:bruhat-transposition-cover}, their complete sets of
lower covers are
\begin{align*}
    \{y\in S_6:y\lessdot142653\}
        &=
        \{124653,132654,142563,142635\},\\
    \{y\in S_6:y\lessdot251436\}
        &=
        \{152436,215436,241536,251346\}.
\end{align*}
These sets are disjoint. Thus this prescribed pair has no common lower
cover.

The global conclusion of Contact Rigidity nevertheless holds. The
admissible cross core \(1256Q\) has lowering \(241635\), and
\[
    142635\lessdot142653,
    \qquad
    142635\lessdot241635,
\]
while
\[
    241536\lessdot251436,
    \qquad
    241536\lessdot241635.
\]
Thus the third maximal singular index \(241635\) has a common
lower cover with each prescribed index, although the prescribed
pair has none.

\end{example}

\paragraph*{Sharpness of the hypotheses.}

\begin{example}[The lci hypothesis cannot be removed]
\label{ex:sharpness-lci}
Let
   $ w=351624$.
The Billey--Warrington classification
\cite[Theorem~1]{BW03} gives
   $ \operatorname{maxsing}(w)
    =
    \{132654,321546\}$.
Both indexing permutations avoid \(3412\) and \(4231\). Hence, by
Theorem~\ref{thm:lakshmibai-sandhya},
    $X_{132654}$
and
    $X_{321546}$
are smooth.

Their complete sets of lower covers are
\begin{align*}
    \{y\in S_6:y\lessdot132654\}
        &=
        \{123654,132564,132645\},\\
    \{y\in S_6:y\lessdot321546\}
        &=
        \{231546,312546,321456\},
\end{align*}
as follows from Lemma~\ref{lem:bruhat-transposition-cover}. These two
sets are disjoint. Since these are the only two maximal singular
components, no pair of distinct elements of
\(\operatorname{maxsing}(w)\) has a common lower cover.

On the other hand, \(w\) itself is the forbidden pattern \(351624\).
Therefore \(X_w\) is not lci by
Theorem~\ref{thm:lci-pattern-criterion}. Thus the lci hypothesis in
Contact Rigidity cannot be omitted.
\end{example}

\begin{example}[One smooth maximal component is not enough]
\label{ex:sharpness-smooth}
Let
    $w=356412$.
A direct check shows that \(w\) avoids the six patterns in
Theorem~\ref{thm:lci-pattern-criterion}, so \(X_w\) is lci. Its
admissible cores are exactly
    $1456Q$
    and
    $2356Q$,
with lowerings
\[
    z(1456Q)=156324,
    \qquad
    z(2356Q)=315426.
\]
Proposition~\ref{prop:core-dictionary} therefore gives
    $\operatorname{maxsing}(w)
    =
    \{156324,315426\}$.

The first maximal component is singular: the entries of \(156324\)
in positions \(2,3,4,6\) are \(5,6,3,4\), and
\(\operatorname{fl}(5,6,3,4)=3412\).
Thus Theorem~\ref{thm:lakshmibai-sandhya} implies that
\(X_{156324}\) is singular. By contrast, \(315426\) avoids both
\(3412\) and \(4231\), so \(X_{315426}\) is smooth.

Their complete sets of lower covers are
\begin{align*}
    \{y\in S_6:y\lessdot156324\}
        &=
        \{136524,146325,153624,154326,156234\},\\
    \{y\in S_6:y\lessdot315426\}
        &=
        \{135426,215436,314526,315246\}.
\end{align*}
These sets are disjoint. Hence the two maximal singular components have
no common lower cover.

Thus the existence of only one smooth maximal singular component does
not force the conclusion of Contact Rigidity. The hypothesis that there
exist two distinct smooth maximal singular components cannot be weakened
to the existence of just one.
\end{example}

\section{The single-branch comparison theorem}
\label{chap:comparison-application}

Throughout this section, \(X_w\) is a Type~A lci Schubert variety, and
all coefficients are rational. We prove that
\[
    \operatorname{Sing}(X_w)=X_z,\quad X_z\text{ smooth}
    \quad\Longrightarrow\quad N_w\simeq IC_z,
\]
and combine this with Contact Rigidity.

Agreement on the dense cell does not determine a perverse sheaf.
For example, if \(j:\mathbb A^1\hookrightarrow\mathbb P^1\) is the
standard open embedding, then \(j_!\mathbb Q_{\mathbb A^1}[1]\) is
perverse because \(j\) is affine
\cite[Corollary~3.5.9, p.~155]{Ach21}. It agrees with
\(IC_{\mathbb P^1}=\mathbb Q_{\mathbb P^1}[1]\) on \(\mathbb A^1\),
but its stalk at infinity is zero. Accordingly, we use
Kazhdan--Lusztig values on every Bruhat stratum of \(X_z\) to exclude
additional composition factors on the boundary.

Section~\ref{sec:comparison-single-branch} gives the geometric
reduction, and Section~\ref{sec:kl-guarded} verifies the hypotheses of
Woo's theorem. Section~\ref{sec:kl-to-comparison} determines \([N_w]\)
from \(P_{u,w}(1)=2\), and hence identifies \(N_w\). The full
Kazhdan--Lusztig formula is then recovered as a consequence.

\subsection{The comparison kernel and the single-branch reduction}
\label{sec:comparison-single-branch}

We use the comparison sequence \eqref{eq:schubert-comparison-sequence}.

\begin{proposition}[Support of the comparison kernel]
\label{prop:comparison-support}
Let \(w\in S_n\), and assume that \(X_w\) is lci. Then
\begin{equation}
    \operatorname{supp}N_w
    =\operatorname{Sing}(X_w)
    =\bigcup_{z\in M_w}X_z.
    \label{eq:comparison-support-singular-locus}
\end{equation}
\end{proposition}

\begin{proof}
By \eqref{eq:support-comparison-kernel}, the complement of
\(\operatorname{supp}N_w\) is the rationally smooth locus of \(X_w\).
In Type~A this is the smooth locus
\eqref{eq:type-a-rational-smoothness}. The final equality follows from
\eqref{eq:singular-locus-components}.
\end{proof}

The following consequence combines Contact Rigidity with the
single-branch theorem proved in Section~\ref{sec:kl-to-comparison}.

\begin{theorem}[Type A Comparison Theorem]
\label{thm:type-a-comparison}
Let \(w\in S_n\), and let
\(X_w\subseteq \operatorname{GL}_n(\mathbb C)/B\) be an lci Schubert
variety. Work with rational coefficients and set
   $ M_w=\operatorname{maxsing}(w)$.
Assume that:
\begin{enumerate}
    \item \(X_z\) is rationally smooth for every \(z\in M_w\);
    \item for every pair of distinct elements \(z,z'\in M_w\), the
    intersection \(X_z\cap X_{z'}\) has codimension at least two in both
    \(X_z\) and \(X_{z'}\).
\end{enumerate}
Here, if \(X_z\cap X_{z'}\) is reducible, ``codimension at least two''
means that every irreducible component of the intersection has
codimension at least two in the specified Schubert variety.
Then
    $|M_w|\leq1$
and the comparison kernel is determined as follows:
\begin{equation}
    N_w\simeq
    \begin{cases}
        0, & M_w=\varnothing,\\[2mm]
        IC_z, & M_w=\{z\}.
    \end{cases}
    \label{eq:type-a-comparison-conclusion}
\end{equation}
\end{theorem}

The geometric conclusion follows directly from Contact Rigidity.

\begin{proposition}[Single-branch reduction]
\label{prop:single-branch-reduction}
Under the hypotheses of Theorem~\ref{thm:type-a-comparison}, one has
    $|M_w|\leq1$.
\end{proposition}

\begin{proof}
By hypothesis~(1) and \eqref{eq:type-a-rational-smoothness}, every
\(X_z\), \(z\in M_w\), is smooth. If \(|M_w|\geq2\), Contact
Rigidity (Theorem~\ref{thm:contact-rigidity-proof}) gives distinct
\(z,z'\in M_w\) with \(y\lessdot z\) and \(y\lessdot z'\). Then
\(X_y\subseteq X_z\cap X_{z'}\) has codimension one in each by
\eqref{eq:cover-codimension-one}, contradicting hypothesis~(2).
\end{proof}

If \(M_w=\varnothing\), then \(X_w\) is smooth by
\eqref{eq:singular-locus-components} and \(N_w=0\).
Until the final proof of Theorem~\ref{thm:type-a-comparison}, we impose
\begin{equation}
    M_w=\{z\},
    \qquad
    \operatorname{Sing}(X_w)=X_z,
    \qquad
    X_z\text{ is smooth}.
    \label{eq:single-branch-hypothesis}
\end{equation}
By Proposition~\ref{prop:comparison-support}, this also gives
    $\operatorname{supp}N_w=X_z$.

\begin{remark}[On the codimension-two hypothesis]
\label{rem:codim-two-hypothesis}

The codimension-two hypothesis is used only in
Proposition~\ref{prop:single-branch-reduction} to exclude multiple
singular components. The single-branch comparison proof uses neither
this hypothesis nor Contact Rigidity. Thus
Theorem~\ref{thm:type-a-comparison} is not a decomposition theorem for
a reducible singular locus; factors on component intersections and
their extension data remain outside its scope.

\end{remark}

\subsection{Guarded obstructions and Woo's theorem}
\label{sec:kl-guarded}

Woo's theorem applies when the singular locus has one irreducible
component and \(w\) avoids
\[
    653421,\quad632541,\quad463152,\quad526413,\quad546213,\quad465132.
\]
It then gives \(P_{u,w}(q)=1+q^h\) for all \(u\leq z\), with some
\(h\geq1\); see \cite[Theorem~1.1 and Corollary~1.2, pp.~2--3]{Woo09}.
For the comparison proof, we only need its value at \(q=1\).

Three of the six patterns are already excluded by the lci criterion.
Indeed, deleting the fourth entry of \(632541\) gives \(63241\), which
has pattern \(53241\). Deleting the third entry of \(463152\) gives
\(46152\), with pattern \(35142\). Deleting the fourth entry of
\(526413\) gives \(52613\), with pattern \(42513\).
Each resulting pattern is forbidden by
Theorem~\ref{thm:lci-pattern-criterion}.

For each remaining pattern, we tighten four of its points to an
admissible core while retaining two fixed guard points. After lowering,
the guards and two core points exhibit a \(3412\) or \(4231\) pattern,
so the Core Dictionary and the smoothness criterion produce a singular
maximal component. The next two lemmas and theorem require only the
lci hypothesis; \eqref{eq:single-branch-hypothesis} is suspended until
Proposition~\ref{prop:single-branch-kl-values}.

\begin{lemma}[Guarded \(P\)-tightening]
\label{lem:guarded-p-tightening}
Let \(w\in S_n\) be lci. Let \(g<e\) be positions, and put \(L=w_e<H=w_g\).
Suppose that a \(4231\)-occurrence in \(w\) has all four positions in
\((g,e)\) and all four values in \((L,H)\).
Then there is an admissible \(P\)-core with the same bounds.
\end{lemma}

\begin{proof}
Among all occurrences satisfying these bounds, choose
\(C=(a,b,c,d)\) with \(b-a\) minimal, and then with \(d-c\) minimal.
The family is finite and nonempty. Write
\(\delta=w_d<\beta=w_b<\gamma=w_c<\alpha=w_a\).
All replacements below remain within the same position and value
bounds.

Suppose that \((p,w_p)\in P_L(C)\), and write \(\eta=w_p\).
Then \(a<p<b\) and \(\beta<\eta<\alpha\).
If \(\eta<\gamma\), replace \(C\) by \((a,p,c,d)\).
If \(\eta>\gamma\), replace it by \((p,b,c,d)\).
Both replacements are \(4231\)-occurrences and decrease the first
gap. Hence \(G(w)\cap P_L(C)=\varnothing\).

Suppose that \((p,w_p)\in P_R(C)\).
Then \(c<p<d\) and \(\delta<\eta=w_p<\gamma\).
If \(\eta<\beta\), use \((a,b,c,p)\); if \(\eta>\beta\), use
\((a,b,p,d)\). Both replacements keep the first gap unchanged and
decrease the second. Thus \(G(w)\cap P_R(C)=\varnothing\).

Finally, suppose that \((p,w_p)\in P_C(C)\).
The values at \(a,b,p,c,d\) are \(\alpha,\beta,\eta,\gamma,\delta\),
where \(\delta<\eta<\alpha\). Their pattern is \(53241\) if
\(\eta<\beta\), \(52341\) if \(\beta<\eta<\gamma\), and \(52431\)
if \(\eta>\gamma\). All three are forbidden by
Theorem~\ref{thm:lci-pattern-criterion}. Hence
\(G(w)\cap P_C(C)=\varnothing\).
The three conditions in Definition~\ref{def:p-core} now hold.
\end{proof}

For a \(3412\)-occurrence \(C=(a,b,c,d)\), write
\(\gamma=w_c<\delta=w_d<\alpha=w_a<\beta=w_b\), and define its
\emph{height} and \emph{amplitude} by
\(\operatorname{ht}(C)=\alpha-\delta\) and
\(\operatorname{amp}(C)=\beta-\gamma\).
Woo uses minimisation of height followed by amplitude in
\cite[Section~4.1, p.~11]{Woo09}.
We use the same choice within a family that is closed under specified
replacements. This allows us to retain the two extra points needed
after lowering.

\begin{lemma}[Stable-family \(Q\)-tightening]
\label{lem:stable-q-tightening}
Let \(w\in S_n\) be lci, and let \(\mathcal Q\) be a nonempty family of \(3412\)-occurrences in
\(w\). For each \(C=(a,b,c,d)\in\mathcal Q\), assume that whenever
the condition in the first column holds, the tuple in the second column
also belongs to \(\mathcal Q\):
\begin{center}
\begin{tabular}{@{}ll@{}}
\toprule
Condition & Required member of \(\mathcal Q\)\\
\midrule
\((p,w_p)\in Q_1(C)\) & \((p,b,c,d)\)\\
\((p,w_p)\in Q_3(C)\) & \((a,b,c,p)\)\\
\((p,w_p)\in Q_2(C)\) & \((a,p,c,d)\)\\
\((p,w_p)\in Q_4(C)\) & \((a,b,p,d)\)\\
\((p,w_p),(q,w_q)\in B(C),\ p<q,\ w_p<w_q\)
    & \((a,b,p,q)\)\\
\bottomrule
\end{tabular}
\end{center}
Then \(\mathcal Q\) contains an admissible \(Q\)-core.
\end{lemma}

\begin{proof}
Choose \(C=(a,b,c,d)\in\mathcal Q\) with minimal height, and then with
minimal amplitude. Both minima exist because the family is finite.
Write \(\gamma=w_c<\delta=w_d<\alpha=w_a<\beta=w_b\).

A point in \(Q_1(C)\) gives a member of height
\(w_p-\delta<\alpha-\delta\). A point in \(Q_3(C)\) gives one of
height \(\alpha-w_p<\alpha-\delta\). Thus neither region contains
a graph point.

A point in \(Q_2(C)\) gives a member with the same height and amplitude
\(w_p-\gamma<\beta-\gamma\). A point in \(Q_4(C)\) gives one with
the same height and amplitude \(\beta-w_p<\beta-\gamma\).
Hence these two regions also contain no graph points.

If the graph points in \(B(C)\) are not decreasing, there are
\(p<q\) in that region with
\(\delta<w_p<w_q<\alpha\).
The last replacement gives a member of height
\(\alpha-w_q<\alpha-\delta\), again a contradiction.
Thus the middle region is reduced.

The lci hypothesis gives
\(G(w)\cap Q_{A_1}(C)=G(w)\cap Q_{A_2}(C)=\varnothing\) by
Lemma~\ref{lem:lci-a-regions}.
All conditions in Definition~\ref{def:q-core} hold, so \(C\) is
admissible.
\end{proof}

\begin{theorem}[Guarded-obstruction theorem]
\label{thm:guarded-obstruction}
Let \(w\in S_n\) be lci. If \(w\) contains \(653421\), \(546213\), or \(465132\), then there
is \(v\in\operatorname{maxsing}(w)\) such that \(X_v\) is singular.
\end{theorem}

\begin{proof}
\emph{The pattern \(653421\).}
Choose occurrence positions \(g<a<b<c<d<e\).
Their values satisfy
\(L:=w_e<w_d<w_b<w_c<w_a<H:=w_g\).
Lemma~\ref{lem:guarded-p-tightening} gives an admissible
\(P\)-core \(C=(a',b',c',d')\) with the same bounds.
Write its values as
\(L<\delta'<\beta'<\gamma'<\alpha'<H\).
The lowering changes the core values from
\((\alpha',\beta',\gamma',\delta')\) to
\((\beta',\delta',\alpha',\gamma')\).
The positions \(g<a'<c'<e\) therefore have values
\(H,\beta',\alpha',L\), which form \(4231\).
The two guard positions \(g,e\) are outside the core and are unchanged.
Thus \(X_{z(C)}\) is singular by
Theorem~\ref{thm:lakshmibai-sandhya}, and
\(z(C)\in\operatorname{maxsing}(w)\) by
Proposition~\ref{prop:core-dictionary}.

\emph{The pattern \(546213\).}
Choose occurrence positions \(g<a<b<c<e<d\), and put
\(L=w_e\), \(H=w_g\). Then
\(L<w_c<w_d<w_a<H<w_b\), so \((a,b,c,d)\) is a
\(3412\)-occurrence.
Let \(\mathcal Q_L\) consist of all \(3412\)-occurrences
\((a',b',c',d')\) satisfying
\begin{equation}
    g<a',\qquad c'<e,\qquad w_{a'}<H,\qquad w_{c'}>L.
    \label{eq:QL-family}
\end{equation}
The initial occurrence belongs to this family. We now check its closure
under the five replacements in Lemma~\ref{lem:stable-q-tightening}.
A \(Q_1\)-replacement moves the first position to the right and
decreases its value, so it preserves \(g<a'\) and \(w_{a'}<H\).
The \(Q_3\)- and \(Q_2\)-replacements change only the fourth and
second positions, respectively, so they do not change the guards.
A \(Q_4\)-replacement moves the third position to the left and
increases its value. The bounds \(c'<e\) and \(w_{c'}>L\) remain valid.
For the last replacement, one has
\(b'<p<q<c'<e\) and \(w_p>w_{d'}>w_{c'}>L\).
Thus the new third position \(p\) also satisfies the guards.
All five replacements stay in \(\mathcal Q_L\).

The lemma gives an admissible \(Q\)-core
\(C=(a',b',c',d')\in\mathcal Q_L\).
Write \(\gamma'=w_{c'}<\delta'=w_{d'}<\alpha'=w_{a'}<\beta'=w_{b'}\).
The guards imply \(g<a'<b'<c'<e\) and
\(L<\gamma'<\alpha'<H\).
Also \(d'\neq e\), since \(w_{d'}>w_{c'}>L=w_e\).
Hence both \(g\) and \(e\) lie outside the core.
After lowering, the positions \(g<a'<b'<e\) have values
\(H,\gamma',\alpha',L\), which form \(4231\).
As in the first case, the smoothness criterion and the Core Dictionary
give a singular maximal component.

\emph{The pattern \(465132\).}
We reduce this case to the preceding one.
For \(v\in S_r\), define
\(\rho(v)_i=r+1-v_{r+1-i}\).
This is rotation of the permutation graph by \(180^\circ\), and
\(\rho^2=\operatorname{id}\).
Rotation takes an occurrence of a pattern \(p\) to an occurrence of
\(\rho(p)\). On the lci forbidden patterns it exchanges
\(53241\) with \(52431\), exchanges \(35142\) with \(42513\), and
fixes \(52341\) and \(351624\).
Thus \(X_{\rho(w)}\) is lci. It also fixes \(3412\) and \(4231\),
so it preserves the smoothness criterion.

For a core \(C=(a,b,c,d)\) in \(w\), its rotated core has positions
\((n+1-d,n+1-c,n+1-b,n+1-a)\) in \(\rho(w)\).
Definitions~\ref{def:p-core} and~\ref{def:q-core} show that rotation
exchanges the following pairs of required regions:
\[
    P_L\leftrightarrow P_R,\qquad
    Q_1\leftrightarrow Q_3,\qquad
    Q_2\leftrightarrow Q_4,\qquad
    Q_{A_1}\leftrightarrow Q_{A_2}.
\]
It takes \(P_C\) to the corresponding central region and \(B(C)\) to
the corresponding middle region. Reversing both coordinates preserves
the decreasing condition in the middle region.
Hence admissibility is preserved. The two lowering formulas also give
\(z(\rho(C))=\rho(z(C))\), where \(\rho(C)\) denotes the rotated core.

Since \(\rho(465132)=546213\), the preceding case, applied to
\(\rho(w)\), produces an admissible core with singular lowering.
Rotating that core back gives an admissible core in \(w\) whose
lowering is singular. The Core Dictionary again gives the required
maximal component.
\end{proof}

\begin{proposition}[Single-branch Kazhdan--Lusztig values]
\label{prop:single-branch-kl-values}
Under \eqref{eq:single-branch-hypothesis}, one has
\begin{equation}
    P_{u,w}(1)=2\qquad(u\leq z).
    \label{eq:single-branch-kl-values}
\end{equation}
\end{proposition}

\begin{proof}
If \(w\) contained \(653421\), \(546213\), or \(465132\),
Theorem~\ref{thm:guarded-obstruction} would give a singular
\(X_v\) with \(v\in M_w\). Since \(M_w=\{z\}\), this would
contradict the smoothness of \(X_z\).
The other three patterns in Woo's theorem were excluded at the start
of this section by the lci criterion.
Thus all six patterns are avoided, and the singular locus has exactly
one irreducible component.
By \cite[Theorem~1.1 and Corollary~1.2, pp.~2--3]{Woo09}, there is an
integer \(h\geq1\) such that \(P_{u,w}(q)=1+q^h\) for every
\(u\leq z\). Evaluation at \(q=1\) proves the claim.
\end{proof}

\subsection{The single-branch comparison theorem}
\label{sec:kl-to-comparison}

Set \(d=\ell(w)\) and \(m=\ell(z)\).
We first use the value at \(e_z\) to determine both the generic rank
and the parity of \(d-m\). The values at all lower fixed points then
exclude additional composition factors.

\begin{proposition}[Generic restriction and parity]
\label{prop:generic-restriction}
Under \eqref{eq:single-branch-hypothesis}, one has
\(N_w|_{C_z}\simeq\mathbb Q_{C_z}[m]\), and \(d-m\) is odd.
\end{proposition}

\begin{proof}
By Proposition~\ref{prop:comparison-support}, \(N_w\) is supported on
\(X_z\). The closed embedding \(X_z\hookrightarrow X_w\) therefore
allows us to regard it as a perverse sheaf on \(X_z\)
\cite[Proposition~3.1.10, p.~133]{Ach21}.
Its restriction to the open cell \(C_z\simeq\mathbb C^m\) is perverse
and constructible for that single stratum, hence is \(L[m]\) for a
finite-rank local system \(L\)
\cite[Lemmas~3.1.3--3.1.4, pp.~130--131]{Ach21}.
Since \(C_z\) is simply connected, \(L\) is constant
\cite[Definitions~1.7.1--1.7.2 and Theorem~1.7.9, pp.~41--44]{Ach21}.

At \(e_z\), equations \eqref{eq:fixed-point-euler-character} and
\eqref{eq:euler-character-comparison-kernel}, together with
Proposition~\ref{prop:single-branch-kl-values}, give
\[
    (-1)^m\operatorname{rk}L
    =\chi_z(N_w)
    =(-1)^d\bigl(1-P_{z,w}(1)\bigr)
    =(-1)^{d+1}.
\]
Taking absolute values gives \(\operatorname{rk}L=1\), and comparing
signs shows that \(d-m\) is odd. Thus
\(N_w|_{C_z}\simeq\mathbb Q_{C_z}[m]\).
\end{proof}

Recall that
\(\mathcal P_w=\operatorname{Perv}_{\mathrm{Bruhat}}(X_w;\mathbb Q)\)
has finite length, with simple objects \(IC_v\), \(v\leq w\)
(Section~\ref{sec:perverse-kl}). Consequently, the classes \([IC_v]\)
form a basis of \(K_0(\mathcal P_w)\), and
\[
    [\mathcal F]=\sum_{v\leq w}m_v(\mathcal F)[IC_v],
    \qquad m_v(\mathcal F)\in\mathbb Z_{\geq0},
\]
where the coefficients are Jordan--H\"older multiplicities
\cite[Definition~A.3.22 and Theorem~A.3.23, p.~479;
Definition~A.9.3, p.~508]{Ach21}.

Write \([e,w]=\{u\in S_n:u\leq w\}\), where \(e\) is the identity
permutation. The fixed-point Euler characteristics from
\eqref{eq:fixed-point-euler-character} are additive in short exact
sequences, so they define a group homomorphism
\begin{equation}
    \chi:K_0(\mathcal P_w)\longrightarrow\mathbb Z^{[e,w]},
    \qquad
    [\mathcal F]\longmapsto
    \bigl(\chi_u(\mathcal F)\bigr)_{u\in[e,w]},
    \label{eq:euler-character-map}
\end{equation}
where \(\mathbb Z^{[e,w]}\) denotes the group of integer-valued
functions on the finite set \([e,w]\).

The following lemma shows that these Euler characteristics determine
the multiplicities \(m_v(\mathcal F)\).

\begin{lemma}[Euler-character triangularity]
\label{lem:euler-character-triangularity}
The homomorphism \(\chi\) is injective.
\end{lemma}

\begin{proof}
The support and normalisation of \(IC_v\) give
\[
    \chi_u(IC_v)=0\quad(u\not\leq v),
    \qquad \chi_v(IC_v)=(-1)^{\ell(v)}.
\]
Indeed, \(\operatorname{supp}IC_v=X_v\), while
\(IC_v|_{C_v}\simeq\mathbb Q_{C_v}[\ell(v)]\).
If \(\alpha=\sum_{v\leq w}a_v[IC_v]\) lies in the kernel of \(\chi\)
and is nonzero, choose a Bruhat-maximal \(v\) with \(a_v\neq0\).
Only indices \(x\geq v\) contribute to \(\chi_v(\alpha)\), and
maximality eliminates all except \(v\). Thus
\[
    0=\chi_v(\alpha)=a_v(-1)^{\ell(v)},
\]
a contradiction. Hence \(\chi\) is injective.
\end{proof}

\begin{theorem}[Single-branch comparison theorem]
\label{thm:single-branch-comparison}
Under \eqref{eq:single-branch-hypothesis}, one has \(N_w\simeq IC_z\).
\end{theorem}

\begin{proof}
For \(u\leq z\), Proposition~\ref{prop:single-branch-kl-values} and
\eqref{eq:euler-character-comparison-kernel} give
\(\chi_u(N_w)=(-1)^{d+1}=(-1)^m\), where the last equality follows
from Proposition~\ref{prop:generic-restriction}.
For \(u\not\leq z\), Proposition~\ref{prop:comparison-support} gives
\(\chi_u(N_w)=0\).
Since \(X_z\) is smooth, \(IC_z\) is
\(\mathbb Q_{X_z}[m]\), viewed on \(X_w\) through the closed
inclusion. Thus
\begin{equation}
    \chi_u(N_w)=\chi_u(IC_z)=
    \begin{cases}
        (-1)^m,&u\leq z,\\
        0,&u\not\leq z.
    \end{cases}
    \label{eq:euler-character-Nw-single-branch}
\end{equation}
Lemma~\ref{lem:euler-character-triangularity} yields
\begin{equation}
    [N_w]=[IC_z]\quad\text{in }K_0(\mathcal P_w).
    \label{eq:single-branch-K0-equality}
\end{equation}

Since the classes \([IC_v]\) form a basis and all composition
multiplicities of \(N_w\) are nonnegative,
\eqref{eq:single-branch-K0-equality} says that \(N_w\) has precisely
one composition factor, namely \(IC_z\). It therefore has length one
and is isomorphic to \(IC_z\). The use of a single simple class is
essential: equality in \(K_0\) alone does not determine extension data.
\end{proof}

The comparison proof used only \(P_{u,w}(1)=2\). We now recover the
full Kazhdan--Lusztig polynomial and express its exponent in terms of
\(\operatorname{codim}_{X_w}X_z\).

\begin{corollary}[Single-branch Kazhdan--Lusztig constancy]
\label{cor:single-branch-kl-constancy}
Under \eqref{eq:single-branch-hypothesis}, one has
\begin{equation}
    P_{u,w}(q)=P_{z,w}(q)
    =1+q^{(\ell(w)-\ell(z)-1)/2}
    \qquad(u\leq z).
    \label{eq:single-branch-kl-constancy}
\end{equation}
In particular, \(\ell(w)-\ell(z)\geq3\).
\end{corollary}

\begin{proof}
By Proposition~\ref{prop:single-branch-kl-values} and
\cite[Corollary~1.2, p.~3]{Woo09}, there is an integer \(h\geq1\)
with \(P_{u,w}(q)=1+q^h\) for every \(u\leq z\).
Put \(k=-d+2h>-d\). The normalisation \eqref{eq:kl-normalization}
gives \(H^k(i_z^*IC_w)\simeq\mathbb Q\).

Apply \(i_z^*\) to the triangle associated with the comparison sequence:
\[
    N_w\longrightarrow\mathbb Q_{X_w}[d]
    \longrightarrow IC_w\longrightarrow N_w[1].
\]
The stalk \(\mathbb Q[d]\) has cohomology only in degree \(-d\),
so its groups in degrees \(k\) and \(k+1\) vanish. The long exact
sequence therefore gives
\[
    H^k(i_z^*IC_w)\simeq H^{k+1}(i_z^*N_w)\neq0.
\]
By Theorem~\ref{thm:single-branch-comparison} and smoothness of \(X_z\),
\(i_z^*N_w\simeq\mathbb Q[m]\). Hence \(k+1=-m\), so
\[
    -d+2h+1=-m,
    \qquad h=\frac{d-m-1}{2}.
\]
Substitution gives the stated formula, and \(h\geq1\) implies
\(d-m\geq3\).
\end{proof}

\begin{proof}[Proof of Theorem~\ref{thm:type-a-comparison}]
Proposition~\ref{prop:single-branch-reduction} gives \(|M_w|\leq1\).
If \(M_w=\varnothing\), then \(X_w\) is smooth and the comparison
morphism is an isomorphism, so \(N_w=0\). If \(M_w=\{z\}\),
the rational smoothness assumption and
\eqref{eq:type-a-rational-smoothness} give that \(X_z\) is smooth.
Theorem~\ref{thm:single-branch-comparison} then gives \(N_w\simeq IC_z\).
\end{proof}

\section{Further questions}
\label{sec:conclusions}

The natural next question is to determine the comparison kernel for a
reducible singular locus. Within the present approach, this requires
Kazhdan--Lusztig information sufficient to identify composition factors
supported on component intersections and smaller strata. Even after
their multiplicities are determined, extension data remain to be
controlled: equality of Grothendieck classes alone does not identify
the object. The single-branch theorem supplies no general splitting
principle, and the codimension hypothesis in
Theorem~\ref{thm:type-a-comparison} excludes the multi-component case
rather than resolving it.

\section*{Acknowledgements}
The author thanks his supervisor, Jay, for his guidance in algebraic
groups and geometric representation theory, his advice on the direction
of this research, and his continuing time, patience, and support.

\section*{Declarations}
\paragraph*{Use of artificial intelligence.}
OpenAI's GPT-5.6 Pro, operating in chat mode, was used to assist the
development, checking, and exposition of the proofs in
Sections~\ref{chap:contact-rigidity} and
\ref{chap:comparison-application}, as well as the design and verification
of the computer-assisted finite classifications and associated code.
AI assistance was also used for editorial condensation and \LaTeX{}
formatting. The author checked the mathematical statements, proofs,
computations, code, outputs, citations, and final text, and accepts full
responsibility for their correctness.

\paragraph*{Data and code availability.}
The finite classification certificates are included in
Appendix~\ref{app:finite-certificates}. The verifier and its recorded
output accompany the article as ancillary files; reproduction
instructions are given in
Appendix~\ref{app:contact-code}. No external dataset is required.

\ifdefempty{\FundingStatement}{}{\paragraph*{Funding.}\FundingStatement}
\ifdefempty{\CompetingInterestsStatement}{}{\paragraph*{Competing interests.}\CompetingInterestsStatement}

\appendix

\section{Finite classification and cross-inheritance certificates}
\label{app:finite-certificates}

This appendix gives the finite certificates used in
Section~\ref{chap:contact-rigidity}. Their completeness follows from
Theorem~\ref{thm:finite-signature-reduction} and
Proposition~\ref{prop:finite-decision}.
A label such as \(1245Q\) denotes a \(Q\)-core on horizontal ranks
\(1,2,4,5\); \(351624@123457\) means that the subsequence on the listed
horizontal ranks flattens to \(351624\). The sets \(U(C)\) and
\(H_{\mathrm{res}}(C)\) are defined in
\eqref{eq:cross-unresolved-residual}.

\Needspace{14\baselineskip}
\subsection*{One-point overlap}

Exactly eight signatures have two locally admissible cores and two
smooth restricted lowerings. Each violates the lci predicate, as certified
below; hence there is no lci survivor.

\begin{center}
\small
\begin{tabular}{@{}llll@{}}
\toprule
\(\pi\) & \(C_1\) & \(C_2\) & forbidden certificate \\
\midrule
4617235 & 1235Q & 2467Q & \(351624@123457\) \\
4617235 & 1236Q & 2457Q & \(351624@123457\) \\
6275134 & 1245P & 1367Q & \(42513@12356\) \\
5741263 & 1245Q & 2367P & \(35142@12467\) \\
3561724 & 1246Q & 3567Q & \(351624@124567\) \\
5267413 & 1256P & 3467Q & \(42513@12367\) \\
4573162 & 1257Q & 3467P & \(35142@13567\) \\
3561724 & 1346Q & 2567Q & \(351624@124567\) \\
\bottomrule
\end{tabular}
\end{center}

\subsection*{Complete two-point cross table}

For each two-point signature, the cross candidates are the sets
\(K\subseteq[6]\) with \(|K|=4\),
\(|K\cap I_1|=|K\cap I_2|=3\), and
\(\operatorname{Core}(\pi,K,Q)\), including the empty-region and
decreasing-chain tests. Select a candidate by minimising \(|U(K)|\),
then breaking ties successively by preferring a restricted lowering
containing \(3412\) or \(4231\), preferring
\(H_{\mathrm{res}}(K)\subseteq H(C_1)\cap H(C_2)\), and taking the
lexicographically smallest position set. This rule makes the selection
reproducible; admissibility follows from the inheritance certificates
and the two extremal repairs.

Tables~\ref{tab:complete-two-point-cross-structural}
and~\ref{tab:complete-two-point-cross-certificates} record all twenty
overlap-two lci signatures and their cross-candidate data.

In Table~\ref{tab:complete-two-point-cross-certificates}, the local
bad-pattern certificate refers to the restricted lowering in the
corresponding row of Table~\ref{tab:complete-two-point-cross-structural};
a dash means that this lowering avoids both \(3412\) and \(4231\).

\begingroup
\small
\renewcommand{\arraystretch}{1.05}
\begin{longtable}{@{}rlllll@{}}
\caption{Two-point signatures, cross candidates, and restricted lowerings.}
\label{tab:complete-two-point-cross-structural}\\
\toprule
No. & \(\pi\) & \(C_1\) & \(C_2\) & \(C\) & lowering \\
\midrule
\endfirsthead
\multicolumn{6}{c}{\tablename~\thetable\ (continued)}\\
\toprule
No. & \(\pi\) & \(C_1\) & \(C_2\) & \(C\) & lowering \\
\midrule
\endhead
\bottomrule
\endfoot
1 & 561234 & 1235Q & 1246Q & 1234Q & 152634 \\
2 & 451623 & 1235Q & 1456Q & 1256Q & 241635 \\
3 & 534612 & 1235P & 2456Q & 2356Q & 513624 \\
4 & 561342 & 1235Q & 2456P & 1245Q & 351462 \\
5 & 451623 & 1235Q & 2456Q & 1256Q & 241635 \\
6 & 561234 & 1236Q & 1245Q & 1234Q & 152634 \\
7 & 562314 & 1236Q & 1345P & 1234Q & 253614 \\
8 & 451623 & 1236Q & 1456Q & 1256Q & 241635 \\
9 & 451623 & 1236Q & 2456Q & 1256Q & 241635 \\
10 & 356124 & 1245Q & 2346Q & 2345Q & 315264 \\
11 & 356124 & 1245Q & 2356Q & 2345Q & 315264 \\
12 & 456123 & 1246Q & 1345Q & 1245Q & 146253 \\
13 & 456123 & 1246Q & 1356Q & 1256Q & 246135 \\
14 & 356124 & 1345Q & 2346Q & 2345Q & 315264 \\
15 & 456123 & 1345Q & 2346Q & 2345Q & 415263 \\
16 & 356124 & 1345Q & 2356Q & 2345Q & 315264 \\
17 & 456123 & 1356Q & 2346Q & 2356Q & 425136 \\
18 & 345612 & 1356Q & 2456Q & 1256Q & 135624 \\
19 & 364512 & 1456Q & 2346P & 3456Q & 361425 \\
20 & 345612 & 1456Q & 2356Q & 1256Q & 135624 \\
\end{longtable}

\Needspace{26\baselineskip}
\begin{longtable}{@{}rllll@{}}
\caption{Local bad patterns and cell certificates for the two-point signatures.}
\label{tab:complete-two-point-cross-certificates}\\
\toprule
No. & local bad pattern & \(U(C)\) & \(H_{\mathrm{res}}(C)\) & inherited? \\
\midrule
\endfirsthead
\multicolumn{5}{c}{\tablename~\thetable\ (continued)}\\
\toprule
No. & local bad pattern & \(U(C)\) & \(H_{\mathrm{res}}(C)\) & inherited? \\
\midrule
\endhead
\bottomrule
\endfoot
1 & \(3412@2456\) & \(\varnothing\) & \(\{\square_{24}\}\) & yes \\
2 & \(\text{--}\) & \(\varnothing\) & \(\varnothing\) & yes \\
3 & \(3412@1456\) & \(\varnothing\) & \(\varnothing\) & yes \\
4 & \(3412@1236\) & \(\varnothing\) & \(\varnothing\) & yes \\
5 & \(\text{--}\) & \(\varnothing\) & \(\varnothing\) & yes \\
6 & \(3412@2456\) & \(\{\square_{33}\}\) & \(\{\square_{24}\}\) & yes \\
7 & \(3412@2456\) & \(\varnothing\) & \(\varnothing\) & yes \\
8 & \(\text{--}\) & \(\varnothing\) & \(\varnothing\) & yes \\
9 & \(\text{--}\) & \(\varnothing\) & \(\varnothing\) & yes \\
10 & \(\text{--}\) & \(\varnothing\) & \(\varnothing\) & yes \\
11 & \(\text{--}\) & \(\varnothing\) & \(\varnothing\) & yes \\
12 & \(3412@2346\) & \(\varnothing\) & \(\{\square_{33}\}\) & yes \\
13 & \(3412@2345\) & \(\varnothing\) & \(\{\square_{33}\}\) & yes \\
14 & \(\text{--}\) & \(\varnothing\) & \(\varnothing\) & yes \\
15 & \(3412@1346\) & \(\varnothing\) & \(\{\square_{33}\}\) & yes \\
16 & \(\text{--}\) & \(\varnothing\) & \(\varnothing\) & yes \\
17 & \(3412@1345\) & \(\varnothing\) & \(\{\square_{33}\}\) & yes \\
18 & \(3412@3456\) & \(\varnothing\) & \(\{\square_{42}\}\) & yes \\
19 & \(3412@1235\) & \(\varnothing\) & \(\varnothing\) & yes \\
20 & \(3412@3456\) & \(\{\square_{33}\}\) & \(\{\square_{42}\}\) & yes \\
\end{longtable}
\endgroup

Only rows~6 and~20 have an unresolved required cell, repaired in
Lemmas~\ref{lem:repair-row-6} and~\ref{lem:repair-row-20}.
The final column of Table~\ref{tab:complete-two-point-cross-certificates}
records \(H_{\mathrm{res}}(C)\subseteq H(C_1)\cap H(C_2)\).

\subsection*{Rectangle-by-rectangle inheritance certificates}

An entry \(\square_{ij}\mapsto Q_2(C_1)/P_R(C_2)\) means that
both named original regions contain the candidate cell; either
containment forces emptiness. The symbol \(\mathsf U\) marks a cell
not inherited from an original required rectangle and occurs only in
rows~6 and~20.

\begingroup
\small
\newenvironment{certificateitems}{%
  \begin{list}{}{%
    \setlength{\leftmargin}{3.8em}%
    \setlength{\labelwidth}{3.2em}%
    \setlength{\labelsep}{0.6em}%
    \setlength{\itemindent}{0pt}%
    \setlength{\topsep}{0.2em}%
    \setlength{\partopsep}{0pt}%
    \setlength{\itemsep}{0pt}%
    \setlength{\parsep}{0pt}%
    \renewcommand{\makelabel}[1]{##1\hfil}}%
}{\end{list}}
\par\noindent\begin{minipage}{\linewidth}
\paragraph*{Row 1: \(561234;\ 1235Q,1246Q\Rightarrow 1234Q\).}
\begin{certificateitems}
\item[\(Q_1\)] \(\square_{12} \mapsto Q_{A_1}(C_1)/Q_{A_1}(C_2)\), \(\square_{13} \mapsto Q_1(C_1)/Q_{A_1}(C_2)\), \(\square_{14} \mapsto Q_1(C_1)/Q_1(C_2)\).
\item[\(Q_2\)] \(\square_{25} \mapsto Q_2(C_1)/Q_2(C_2)\).
\item[\(Q_3\)] \(\square_{32} \mapsto Q_4(C_2)\), \(\square_{33} \mapsto Q_3(C_1)/Q_4(C_2)\), \(\square_{34} \mapsto Q_3(C_1)\).
\item[\(Q_4\)] \(\square_{21} \mapsto Q_4(C_1)\).
\item[\(Q_{A_1}\)] \(\square_{11} \mapsto Q_{A_1}(C_1)\).
\item[\(Q_{A_2}\)] \(\square_{35} \mapsto Q_{A_2}(C_1)/Q_2(C_2)\).
\end{certificateitems}
\end{minipage}\par\smallskip

\par\noindent\begin{minipage}{\linewidth}
\paragraph*{Row 2: \(451623;\ 1235Q,1456Q\Rightarrow 1256Q\).}
\begin{certificateitems}
\item[\(Q_1\)] \(\square_{13} \mapsto Q_1(C_1)/Q_1(C_2)\).
\item[\(Q_2\)] \(\square_{24} \mapsto Q_2(C_1)\), \(\square_{34} \mapsto Q_{A_2}(C_1)\), \(\square_{44} \mapsto Q_{A_2}(C_1)/Q_2(C_2)\).
\item[\(Q_3\)] \(\square_{53} \mapsto Q_3(C_2)\).
\item[\(Q_4\)] \(\square_{22} \mapsto Q_{A_1}(C_2)\), \(\square_{32} \mapsto Q_3(C_1)/Q_{A_1}(C_2)\), \(\square_{42} \mapsto Q_3(C_1)/Q_4(C_2)\).
\item[\(Q_{A_1}\)] \(\square_{12} \mapsto Q_1(C_1)/Q_{A_1}(C_2)\).
\item[\(Q_{A_2}\)] \(\square_{54} \mapsto Q_{A_2}(C_2)\).
\end{certificateitems}
\end{minipage}\par\smallskip

\par\noindent\begin{minipage}{\linewidth}
\paragraph*{Row 3: \(534612;\ 1235P,2456Q\Rightarrow 2356Q\).}
\begin{certificateitems}
\item[\(Q_1\)] \(\square_{22} \mapsto P_C(C_1)/Q_1(C_2)\).
\item[\(Q_2\)] \(\square_{33} \mapsto P_R(C_1)\), \(\square_{43} \mapsto P_R(C_1)/Q_2(C_2)\).
\item[\(Q_3\)] \(\square_{52} \mapsto Q_3(C_2)\).
\item[\(Q_4\)] \(\square_{31} \mapsto P_R(C_1)/Q_{A_1}(C_2)\), \(\square_{41} \mapsto P_R(C_1)/Q_4(C_2)\).
\item[\(Q_{A_1}\)] \(\square_{21} \mapsto P_C(C_1)/Q_{A_1}(C_2)\).
\item[\(Q_{A_2}\)] \(\square_{53} \mapsto Q_{A_2}(C_2)\).
\end{certificateitems}
\end{minipage}\par\smallskip

\par\noindent\begin{minipage}{\linewidth}
\paragraph*{Row 4: \(561342;\ 1235Q,2456P\Rightarrow 1245Q\).}
\begin{certificateitems}
\item[\(Q_1\)] \(\square_{14} \mapsto Q_1(C_1)\).
\item[\(Q_2\)] \(\square_{25} \mapsto Q_2(C_1)/P_L(C_2)\), \(\square_{35} \mapsto Q_{A_2}(C_1)/P_L(C_2)\).
\item[\(Q_3\)] \(\square_{44} \mapsto Q_3(C_1)/P_C(C_2)\).
\item[\(Q_4\)] \(\square_{23} \mapsto Q_4(C_1)/P_L(C_2)\), \(\square_{33} \mapsto P_L(C_2)\).
\item[\(Q_{A_1}\)] \(\square_{13} \mapsto Q_{A_1}(C_1)\).
\item[\(Q_{A_2}\)] \(\square_{45} \mapsto Q_{A_2}(C_1)/P_C(C_2)\).
\end{certificateitems}
\end{minipage}\par\smallskip

\par\noindent\begin{minipage}{\linewidth}
\paragraph*{Row 5: \(451623;\ 1235Q,2456Q\Rightarrow 1256Q\).}
\begin{certificateitems}
\item[\(Q_1\)] \(\square_{13} \mapsto Q_1(C_1)\).
\item[\(Q_2\)] \(\square_{24} \mapsto Q_2(C_1)/Q_1(C_2)\), \(\square_{34} \mapsto Q_{A_2}(C_1)/Q_1(C_2)\), \(\square_{44} \mapsto Q_{A_2}(C_1)\).
\item[\(Q_3\)] \(\square_{53} \mapsto Q_3(C_2)\).
\item[\(Q_4\)] \(\square_{22} \mapsto Q_{A_1}(C_2)\), \(\square_{32} \mapsto Q_3(C_1)/Q_{A_1}(C_2)\), \(\square_{42} \mapsto Q_3(C_1)/Q_4(C_2)\).
\item[\(Q_{A_1}\)] \(\square_{12} \mapsto Q_1(C_1)\).
\item[\(Q_{A_2}\)] \(\square_{54} \mapsto Q_3(C_2)\).
\end{certificateitems}
\end{minipage}\par\smallskip

\par\noindent\begin{minipage}{\linewidth}
\paragraph*{Row 6: \(561234;\ 1236Q,1245Q\Rightarrow 1234Q\).}
\begin{certificateitems}
\item[\(Q_1\)] \(\square_{12} \mapsto Q_{A_1}(C_1)/Q_{A_1}(C_2)\), \(\square_{13} \mapsto Q_{A_1}(C_1)/Q_1(C_2)\), \(\square_{14} \mapsto Q_1(C_1)/Q_1(C_2)\).
\item[\(Q_2\)] \(\square_{25} \mapsto Q_2(C_1)/Q_2(C_2)\).
\item[\(Q_3\)] \(\square_{32} \mapsto Q_4(C_2)\), \(\square_{33} \mapsto\mathsf{U}\), \(\square_{34} \mapsto Q_3(C_1)\).
\item[\(Q_4\)] \(\square_{21} \mapsto Q_4(C_1)\).
\item[\(Q_{A_1}\)] \(\square_{11} \mapsto Q_{A_1}(C_1)\).
\item[\(Q_{A_2}\)] \(\square_{35} \mapsto Q_{A_2}(C_1)/Q_2(C_2)\).
\end{certificateitems}
\end{minipage}\par\smallskip
The unresolved cell is repaired by the leftmost construction in
Lemma~\ref{lem:repair-row-6}.

\par\noindent\begin{minipage}{\linewidth}
\paragraph*{Row 7: \(562314;\ 1236Q,1345P\Rightarrow 1234Q\).}
\begin{certificateitems}
\item[\(Q_1\)] \(\square_{13} \mapsto Q_{A_1}(C_1)/P_L(C_2)\), \(\square_{14} \mapsto Q_1(C_1)/P_L(C_2)\).
\item[\(Q_2\)] \(\square_{25} \mapsto Q_2(C_1)\).
\item[\(Q_3\)] \(\square_{33} \mapsto P_C(C_2)\), \(\square_{34} \mapsto Q_3(C_1)/P_C(C_2)\).
\item[\(Q_4\)] \(\square_{22} \mapsto Q_4(C_1)/P_L(C_2)\).
\item[\(Q_{A_1}\)] \(\square_{12} \mapsto Q_{A_1}(C_1)/P_L(C_2)\).
\item[\(Q_{A_2}\)] \(\square_{35} \mapsto Q_{A_2}(C_1)\).
\end{certificateitems}
\end{minipage}\par\smallskip

\par\noindent\begin{minipage}{\linewidth}
\paragraph*{Row 8: \(451623;\ 1236Q,1456Q\Rightarrow 1256Q\).}
\begin{certificateitems}
\item[\(Q_1\)] \(\square_{13} \mapsto Q_1(C_1)/Q_1(C_2)\).
\item[\(Q_2\)] \(\square_{24} \mapsto Q_2(C_1)\), \(\square_{34} \mapsto Q_{A_2}(C_1)\), \(\square_{44} \mapsto Q_{A_2}(C_1)/Q_2(C_2)\).
\item[\(Q_3\)] \(\square_{53} \mapsto Q_3(C_1)/Q_3(C_2)\).
\item[\(Q_4\)] \(\square_{22} \mapsto Q_4(C_1)/Q_{A_1}(C_2)\), \(\square_{32} \mapsto Q_{A_1}(C_2)\), \(\square_{42} \mapsto Q_4(C_2)\).
\item[\(Q_{A_1}\)] \(\square_{12} \mapsto Q_{A_1}(C_1)/Q_{A_1}(C_2)\).
\item[\(Q_{A_2}\)] \(\square_{54} \mapsto Q_{A_2}(C_1)/Q_{A_2}(C_2)\).
\end{certificateitems}
\end{minipage}\par\smallskip

\par\noindent\begin{minipage}{\linewidth}
\paragraph*{Row 9: \(451623;\ 1236Q,2456Q\Rightarrow 1256Q\).}
\begin{certificateitems}
\item[\(Q_1\)] \(\square_{13} \mapsto Q_1(C_1)\).
\item[\(Q_2\)] \(\square_{24} \mapsto Q_2(C_1)/Q_1(C_2)\), \(\square_{34} \mapsto Q_{A_2}(C_1)/Q_1(C_2)\), \(\square_{44} \mapsto Q_{A_2}(C_1)\).
\item[\(Q_3\)] \(\square_{53} \mapsto Q_3(C_1)/Q_3(C_2)\).
\item[\(Q_4\)] \(\square_{22} \mapsto Q_4(C_1)/Q_{A_1}(C_2)\), \(\square_{32} \mapsto Q_{A_1}(C_2)\), \(\square_{42} \mapsto Q_4(C_2)\).
\item[\(Q_{A_1}\)] \(\square_{12} \mapsto Q_{A_1}(C_1)\).
\item[\(Q_{A_2}\)] \(\square_{54} \mapsto Q_{A_2}(C_1)/Q_3(C_2)\).
\end{certificateitems}
\end{minipage}\par\smallskip

\par\noindent\begin{minipage}{\linewidth}
\paragraph*{Row 10: \(356124;\ 1245Q,2346Q\Rightarrow 2345Q\).}
\begin{certificateitems}
\item[\(Q_1\)] \(\square_{22} \mapsto Q_{A_1}(C_2)\), \(\square_{23} \mapsto Q_2(C_1)/Q_{A_1}(C_2)\), \(\square_{24} \mapsto Q_2(C_1)/Q_1(C_2)\).
\item[\(Q_2\)] \(\square_{35} \mapsto Q_2(C_2)\).
\item[\(Q_3\)] \(\square_{42} \mapsto Q_3(C_1)\), \(\square_{43} \mapsto Q_{A_2}(C_1)\), \(\square_{44} \mapsto Q_{A_2}(C_1)/Q_3(C_2)\).
\item[\(Q_4\)] \(\square_{31} \mapsto Q_4(C_1)/Q_4(C_2)\).
\item[\(Q_{A_1}\)] \(\square_{21} \mapsto Q_4(C_1)/Q_{A_1}(C_2)\).
\item[\(Q_{A_2}\)] \(\square_{45} \mapsto Q_{A_2}(C_2)\).
\end{certificateitems}
\end{minipage}\par\smallskip

\par\noindent\begin{minipage}{\linewidth}
\paragraph*{Row 11: \(356124;\ 1245Q,2356Q\Rightarrow 2345Q\).}
\begin{certificateitems}
\item[\(Q_1\)] \(\square_{22} \mapsto Q_{A_1}(C_2)\), \(\square_{23} \mapsto Q_2(C_1)/Q_{A_1}(C_2)\), \(\square_{24} \mapsto Q_2(C_1)/Q_1(C_2)\).
\item[\(Q_2\)] \(\square_{35} \mapsto Q_2(C_2)\).
\item[\(Q_3\)] \(\square_{42} \mapsto Q_3(C_1)/Q_4(C_2)\), \(\square_{43} \mapsto Q_{A_2}(C_1)/Q_4(C_2)\), \(\square_{44} \mapsto Q_{A_2}(C_1)\).
\item[\(Q_4\)] \(\square_{31} \mapsto Q_4(C_1)\).
\item[\(Q_{A_1}\)] \(\square_{21} \mapsto Q_4(C_1)\).
\item[\(Q_{A_2}\)] \(\square_{45} \mapsto Q_2(C_2)\).
\end{certificateitems}
\end{minipage}\par\smallskip

\par\noindent\begin{minipage}{\linewidth}
\paragraph*{Row 12: \(456123;\ 1246Q,1345Q\Rightarrow 1245Q\).}
\begin{certificateitems}
\item[\(Q_1\)] \(\square_{12} \mapsto Q_{A_1}(C_1)/Q_1(C_2)\), \(\square_{13} \mapsto Q_1(C_1)/Q_1(C_2)\).
\item[\(Q_2\)] \(\square_{24} \mapsto Q_2(C_1)\), \(\square_{34} \mapsto Q_2(C_1)/Q_2(C_2)\).
\item[\(Q_3\)] \(\square_{42} \mapsto Q_3(C_2)\), \(\square_{43} \mapsto Q_3(C_1)/Q_3(C_2)\).
\item[\(Q_4\)] \(\square_{21} \mapsto Q_4(C_1)/Q_{A_1}(C_2)\), \(\square_{31} \mapsto Q_4(C_1)/Q_4(C_2)\).
\item[\(Q_{A_1}\)] \(\square_{11} \mapsto Q_{A_1}(C_1)/Q_{A_1}(C_2)\).
\item[\(Q_{A_2}\)] \(\square_{44} \mapsto Q_{A_2}(C_1)/Q_{A_2}(C_2)\).
\end{certificateitems}
\end{minipage}\par\smallskip

\par\noindent\begin{minipage}{\linewidth}
\paragraph*{Row 13: \(456123;\ 1246Q,1356Q\Rightarrow 1256Q\).}
\begin{certificateitems}
\item[\(Q_1\)] \(\square_{13} \mapsto Q_1(C_1)/Q_1(C_2)\).
\item[\(Q_2\)] \(\square_{24} \mapsto Q_2(C_1)\), \(\square_{34} \mapsto Q_2(C_1)/Q_2(C_2)\), \(\square_{44} \mapsto Q_{A_2}(C_1)/Q_2(C_2)\).
\item[\(Q_3\)] \(\square_{53} \mapsto Q_3(C_1)/Q_3(C_2)\).
\item[\(Q_4\)] \(\square_{22} \mapsto Q_4(C_1)/Q_{A_1}(C_2)\), \(\square_{32} \mapsto Q_4(C_1)/Q_4(C_2)\), \(\square_{42} \mapsto Q_4(C_2)\).
\item[\(Q_{A_1}\)] \(\square_{12} \mapsto Q_{A_1}(C_1)/Q_{A_1}(C_2)\).
\item[\(Q_{A_2}\)] \(\square_{54} \mapsto Q_{A_2}(C_1)/Q_{A_2}(C_2)\).
\end{certificateitems}
\end{minipage}\par\smallskip

\par\noindent\begin{minipage}{\linewidth}
\paragraph*{Row 14: \(356124;\ 1345Q,2346Q\Rightarrow 2345Q\).}
\begin{certificateitems}
\item[\(Q_1\)] \(\square_{22} \mapsto Q_1(C_1)/Q_{A_1}(C_2)\), \(\square_{23} \mapsto Q_{A_1}(C_2)\), \(\square_{24} \mapsto Q_1(C_2)\).
\item[\(Q_2\)] \(\square_{35} \mapsto Q_2(C_1)/Q_2(C_2)\).
\item[\(Q_3\)] \(\square_{42} \mapsto Q_3(C_1)\), \(\square_{43} \mapsto Q_{A_2}(C_1)\), \(\square_{44} \mapsto Q_{A_2}(C_1)/Q_3(C_2)\).
\item[\(Q_4\)] \(\square_{31} \mapsto Q_4(C_1)/Q_4(C_2)\).
\item[\(Q_{A_1}\)] \(\square_{21} \mapsto Q_{A_1}(C_1)/Q_{A_1}(C_2)\).
\item[\(Q_{A_2}\)] \(\square_{45} \mapsto Q_{A_2}(C_1)/Q_{A_2}(C_2)\).
\end{certificateitems}
\end{minipage}\par\smallskip

\par\noindent\begin{minipage}{\linewidth}
\paragraph*{Row 15: \(456123;\ 1345Q,2346Q\Rightarrow 2345Q\).}
\begin{certificateitems}
\item[\(Q_1\)] \(\square_{22} \mapsto Q_1(C_1)/Q_{A_1}(C_2)\), \(\square_{23} \mapsto Q_1(C_1)/Q_1(C_2)\), \(\square_{24} \mapsto Q_1(C_2)\).
\item[\(Q_2\)] \(\square_{35} \mapsto Q_2(C_1)/Q_2(C_2)\).
\item[\(Q_3\)] \(\square_{42} \mapsto Q_3(C_1)\), \(\square_{43} \mapsto Q_3(C_1)/Q_3(C_2)\), \(\square_{44} \mapsto Q_{A_2}(C_1)/Q_3(C_2)\).
\item[\(Q_4\)] \(\square_{31} \mapsto Q_4(C_1)/Q_4(C_2)\).
\item[\(Q_{A_1}\)] \(\square_{21} \mapsto Q_{A_1}(C_1)/Q_{A_1}(C_2)\).
\item[\(Q_{A_2}\)] \(\square_{45} \mapsto Q_{A_2}(C_1)/Q_{A_2}(C_2)\).
\end{certificateitems}
\end{minipage}\par\smallskip

\par\noindent\begin{minipage}{\linewidth}
\paragraph*{Row 16: \(356124;\ 1345Q,2356Q\Rightarrow 2345Q\).}
\begin{certificateitems}
\item[\(Q_1\)] \(\square_{22} \mapsto Q_1(C_1)/Q_{A_1}(C_2)\), \(\square_{23} \mapsto Q_{A_1}(C_2)\), \(\square_{24} \mapsto Q_1(C_2)\).
\item[\(Q_2\)] \(\square_{35} \mapsto Q_2(C_1)/Q_2(C_2)\).
\item[\(Q_3\)] \(\square_{42} \mapsto Q_3(C_1)/Q_4(C_2)\), \(\square_{43} \mapsto Q_{A_2}(C_1)/Q_4(C_2)\), \(\square_{44} \mapsto Q_{A_2}(C_1)\).
\item[\(Q_4\)] \(\square_{31} \mapsto Q_4(C_1)\).
\item[\(Q_{A_1}\)] \(\square_{21} \mapsto Q_{A_1}(C_1)\).
\item[\(Q_{A_2}\)] \(\square_{45} \mapsto Q_{A_2}(C_1)/Q_2(C_2)\).
\end{certificateitems}
\end{minipage}\par\smallskip

\par\noindent\begin{minipage}{\linewidth}
\paragraph*{Row 17: \(456123;\ 1356Q,2346Q\Rightarrow 2356Q\).}
\begin{certificateitems}
\item[\(Q_1\)] \(\square_{23} \mapsto Q_1(C_1)/Q_1(C_2)\), \(\square_{24} \mapsto Q_1(C_2)\).
\item[\(Q_2\)] \(\square_{35} \mapsto Q_2(C_1)/Q_2(C_2)\), \(\square_{45} \mapsto Q_2(C_1)/Q_{A_2}(C_2)\).
\item[\(Q_3\)] \(\square_{53} \mapsto Q_3(C_1)/Q_3(C_2)\), \(\square_{54} \mapsto Q_{A_2}(C_1)/Q_3(C_2)\).
\item[\(Q_4\)] \(\square_{32} \mapsto Q_4(C_1)/Q_4(C_2)\), \(\square_{42} \mapsto Q_4(C_1)\).
\item[\(Q_{A_1}\)] \(\square_{22} \mapsto Q_{A_1}(C_1)/Q_{A_1}(C_2)\).
\item[\(Q_{A_2}\)] \(\square_{55} \mapsto Q_{A_2}(C_1)/Q_{A_2}(C_2)\).
\end{certificateitems}
\end{minipage}\par\smallskip

\par\noindent\begin{minipage}{\linewidth}
\paragraph*{Row 18: \(345612;\ 1356Q,2456Q\Rightarrow 1256Q\).}
\begin{certificateitems}
\item[\(Q_1\)] \(\square_{12} \mapsto Q_1(C_1)\).
\item[\(Q_2\)] \(\square_{23} \mapsto Q_1(C_2)\), \(\square_{33} \mapsto Q_2(C_1)/Q_1(C_2)\), \(\square_{43} \mapsto Q_2(C_1)\).
\item[\(Q_3\)] \(\square_{52} \mapsto Q_3(C_1)/Q_3(C_2)\).
\item[\(Q_4\)] \(\square_{21} \mapsto Q_{A_1}(C_1)/Q_{A_1}(C_2)\), \(\square_{31} \mapsto Q_4(C_1)/Q_{A_1}(C_2)\), \(\square_{41} \mapsto Q_4(C_1)/Q_4(C_2)\).
\item[\(Q_{A_1}\)] \(\square_{11} \mapsto Q_{A_1}(C_1)\).
\item[\(Q_{A_2}\)] \(\square_{53} \mapsto Q_{A_2}(C_1)/Q_3(C_2)\).
\end{certificateitems}
\end{minipage}\par\smallskip

\par\noindent\begin{minipage}{\linewidth}
\paragraph*{Row 19: \(364512;\ 1456Q,2346P\Rightarrow 3456Q\).}
\begin{certificateitems}
\item[\(Q_1\)] \(\square_{32} \mapsto Q_1(C_1)/P_C(C_2)\), \(\square_{33} \mapsto P_C(C_2)\).
\item[\(Q_2\)] \(\square_{44} \mapsto Q_2(C_1)/P_R(C_2)\).
\item[\(Q_3\)] \(\square_{52} \mapsto Q_3(C_1)/P_R(C_2)\), \(\square_{53} \mapsto Q_{A_2}(C_1)/P_R(C_2)\).
\item[\(Q_4\)] \(\square_{41} \mapsto Q_4(C_1)\).
\item[\(Q_{A_1}\)] \(\square_{31} \mapsto Q_{A_1}(C_1)\).
\item[\(Q_{A_2}\)] \(\square_{54} \mapsto Q_{A_2}(C_1)/P_R(C_2)\).
\end{certificateitems}
\end{minipage}\par\smallskip

\par\noindent\begin{minipage}{\linewidth}
\paragraph*{Row 20: \(345612;\ 1456Q,2356Q\Rightarrow 1256Q\).}
\begin{certificateitems}
\item[\(Q_1\)] \(\square_{12} \mapsto Q_1(C_1)\).
\item[\(Q_2\)] \(\square_{23} \mapsto Q_1(C_2)\), \(\square_{33} \mapsto\mathsf{U}\), \(\square_{43} \mapsto Q_2(C_1)\).
\item[\(Q_3\)] \(\square_{52} \mapsto Q_3(C_1)/Q_3(C_2)\).
\item[\(Q_4\)] \(\square_{21} \mapsto Q_{A_1}(C_1)/Q_{A_1}(C_2)\), \(\square_{31} \mapsto Q_{A_1}(C_1)/Q_4(C_2)\), \(\square_{41} \mapsto Q_4(C_1)/Q_4(C_2)\).
\item[\(Q_{A_1}\)] \(\square_{11} \mapsto Q_{A_1}(C_1)\).
\item[\(Q_{A_2}\)] \(\square_{53} \mapsto Q_{A_2}(C_1)/Q_3(C_2)\).
\end{certificateitems}
\end{minipage}\par\smallskip
The unresolved cell is repaired by the rightmost construction in
Lemma~\ref{lem:repair-row-20}.

\endgroup

\section{Ancillary program and recorded output}
\label{app:contact-code}

The accompanying files \texttt{signature\_verify.py} and
\texttt{signature\_verify\_output.txt} contain the finite verifier and
its recorded output. They are supplied in the \texttt{anc/} directory
with the arXiv source submission.
The program uses Python~3 and the standard-library module
\texttt{itertools}; no third-party packages are required. From that
directory, run
\begin{quote}\ttfamily\small
python3 signature\_verify.py > reproduced\_output.txt
\end{quote}
and compare the result with \texttt{signature\_verify\_output.txt}.
Positions are zero-based internally and one-based in the printed core
labels. In the output, \texttt{U}, \texttt{H}, and
\texttt{H-inherited} denote \(U(C)\), \(H_{\mathrm{res}}(C)\), and
the inclusion \(H_{\mathrm{res}}(C)\subseteq H(C_1)\cap H(C_2)\).
A printed cell pair \texttt{(h,v)} represents the elementary cell
\(\square_{hv}\) of \eqref{eq:elementary-cell}.

The program starts from the core definitions, pattern-avoidance
predicates, and deterministic cross-candidate rule; no surviving row
is hard-coded. It also checks explicitly that no selected cross
candidate has a noncore union point in its middle region, as required
by Theorem~\ref{thm:complete-cross-table}.
The recorded output reports eight one-point signatures
before the lci test and none after it, together with twenty two-point
lci signatures and their inheritance data. Its role in the proof is
justified by Theorem~\ref{thm:finite-signature-reduction} and
Proposition~\ref{prop:finite-decision}, which establish the finite
reduction and completeness for arbitrary ambient rank.

\bibliographystyle{amsplain}
\bibliography{references}

@book{BL00,
  author    = {Billey, Sara and Lakshmibai, V.},
  title     = {Singular Loci of {Schubert} Varieties},
  series    = {Progress in Mathematics},
  volume    = {182},
  publisher = {Birkh{\"a}user Boston},
  address   = {Boston, MA},
  edition   = {1st},
  year      = {2000},
  doi       = {10.1007/978-1-4612-1324-6}
}

@book{BB05,
  author    = {Bj{\"o}rner, Anders and Brenti, Francesco},
  title     = {Combinatorics of {Coxeter} Groups},
  series    = {Graduate Texts in Mathematics},
  volume    = {231},
  publisher = {Springer},
  address   = {Berlin and Heidelberg},
  edition   = {1st},
  year      = {2005},
  doi       = {10.1007/3-540-27596-7}
}

@article{LS90,
  author  = {Lakshmibai, V. and Sandhya, B.},
  title   = {Criterion for Smoothness of {Schubert} Varieties in {$SL(n)/B$}},
  journal = {Proceedings of the Indian Academy of Sciences. Mathematical Sciences},
  volume  = {100},
  number  = {1},
  pages   = {45--52},
  year    = {1990},
  doi     = {10.1007/BF02881113}
}

@article{WY08,
  author  = {Woo, Alexander and Yong, Alexander},
  title   = {Governing Singularities of {Schubert} Varieties},
  journal = {Journal of Algebra},
  volume  = {320},
  number  = {2},
  pages   = {495--520},
  year    = {2008},
  doi     = {10.1016/j.jalgebra.2007.12.016}
}

@article{BW03,
  author  = {Billey, Sara C. and Warrington, Gregory S.},
  title   = {Maximal Singular Loci of {Schubert} Varieties in {$SL(n)/B$}},
  journal = {Transactions of the American Mathematical Society},
  volume  = {355},
  number  = {10},
  pages   = {3915--3945},
  year    = {2003},
  doi     = {10.1090/S0002-9947-03-03019-8}
}

@article{UW13,
  author  = {{\'U}lfarsson, Henning and Woo, Alexander},
  title   = {Which {Schubert} Varieties Are Local Complete Intersections?},
  journal = {Proceedings of the London Mathematical Society},
  volume  = {107},
  number  = {5},
  pages   = {1004--1052},
  year    = {2013},
  doi     = {10.1112/plms/pdt004}
}

@article{Woo09,
  author  = {Woo, Alexander},
  title   = {Permutations with {Kazhdan--Lusztig} Polynomial
             {$P_{id,w}(q)=1+q^h$}},
  journal = {The Electronic Journal of Combinatorics},
  volume  = {16},
  number  = {2},
  eid     = {R10},
  pages   = {R10},
  year    = {2009},
  note    = {With an appendix by Sara Billey and Jonathan Weed},
  doi     = {10.37236/76}
}

@article{CK03,
  author  = {Carrell, James B. and Kuttler, Jochen},
  title   = {Smooth Points of {$T$}-Stable Varieties in {$G/B$}
             and the {Peterson} Map},
  journal = {Inventiones Mathematicae},
  volume  = {151},
  number  = {2},
  pages   = {353--379},
  year    = {2003},
  doi     = {10.1007/s00222-002-0256-5}
}

@article{Hepler19,
  author  = {Hepler, Brian},
  title   = {Rational Homology Manifolds and Hypersurface Normalizations},
  journal = {Proceedings of the American Mathematical Society},
  volume  = {147},
  number  = {4},
  pages   = {1605--1613},
  year    = {2019},
  doi     = {10.1090/proc/14391}
}

@article{BBD82,
  author  = {Beilinson, A. A. and Bernstein, J. and Deligne, P.},
  title   = {Faisceaux pervers},
  journal = {Ast{\'e}risque},
  volume  = {100},
  pages   = {5--171},
  year    = {1982}
}

@article{KL79,
  author  = {Kazhdan, David and Lusztig, George},
  title   = {Representations of {Coxeter} Groups and {Hecke} Algebras},
  journal = {Inventiones Mathematicae},
  volume  = {53},
  number  = {2},
  pages   = {165--184},
  year    = {1979},
  doi     = {10.1007/BF01390031}
}

@book{Ach21,
  author    = {Achar, Pramod N.},
  title     = {Perverse Sheaves and Applications to Representation Theory},
  series    = {Mathematical Surveys and Monographs},
  volume    = {258},
  publisher = {American Mathematical Society},
  address   = {Providence, RI},
  year      = {2021},
  doi       = {10.1090/surv/258}
}
\end{document}